\documentclass[11pt,letterpaper]{amsart}
\usepackage[utf8]{inputenc}
\usepackage[utf8]{inputenc}
\usepackage{amsmath,amsthm,amssymb}
\usepackage{xcolor}
\usepackage[]{graphicx}
\usepackage{caption}
\usepackage{subfig}
\usepackage{float}
\usepackage{comment}

\newcommand{\bR}{\mathbb{R}}

\newcommand{\bC}{\mathbb{C}}

\DeclareMathOperator{\MA}{MA}
\DeclareMathOperator{\SL}{SL}
\DeclareMathOperator{\GL}{GL}
\DeclareMathOperator{\PSH}{PSH}

\DeclareMathOperator{\conv}{conv}

\DeclareMathOperator{\affspan}{affine span}

\DeclareMathOperator{\St}{St}
\DeclareMathOperator{\SmSt}{SmSt}
\DeclareMathOperator{\sgn}{sgn}
\DeclareMathOperator{\Log}{Log}
\DeclareMathOperator{\dist}{dist}

\newtheorem{lemma}{Lemma}
\newtheorem{theorem}{Theorem}

\newtheorem{corollary}{Corollary}

\newtheorem{definition}{Definition}\theoremstyle{definition}
\newtheorem{remark}{Remark}\theoremstyle{definition}
\newtheorem{example}{Example}

\title[Monge-Ampère equations of reflexive polytopes]{Solvability of Monge-Ampère equations and tropical affine structures on reflexive polytopes}
\author{Rolf Andreasson \and Jakob Hultgren}
\date{\today}

\address{Dept of Mathematical Sciences\\
  Chalmers University of Technology\\
  Göteborg, Sweden}
  \email{rolfan@chalmers.se}
  
\address{Dept of Mathematics and Mathematical Statistics\\
  Umeå University\\ 
  901 87 Umeå, Sweden}
\email{jakob.hultgren@umu.se}

\begin{document}

\begin{abstract}
    Given a reflexive polytope with a height function, we prove a necessary and sufficient condition for solvability of the associated Monge-Ampère equation. When the polytope is Delzant, solvability of this equation implies 
    the metric SYZ conjecture for the corresponding family of Calabi-Yau hypersurfaces. 
    We show how the location of the singularities in the tropical affine structure is determined by the PDE in the spirit of a free boundary problem and give positive and negative examples, demonstrating subtle issues with both solvability and properties of the singular set. We also improve on existing results regarding the SYZ conjecture for the Fermat family by showing regularity of the limiting potential.
 \end{abstract}
\maketitle
\section{Introduction}
Let $Y$ be the toric variety defined by a reflexive polytope $\Delta\subset M_\mathbb R$, where $M$ is the character lattice of $Y$ and $M_\mathbb R = M\otimes \mathbb R$. Let $d\geq 1$ and assume $\dim Y = d+1$. Fixing a \emph{height function} $h:\Delta\cap M\rightarrow \mathbb Z$ such that $h(0)=0<h(m)$ for any $m\in \Delta\cap M\setminus \{0\}$ we get a family of hypersurfaces in $Y$ 
$$ X = \{(x,t)\in Y\times \mathbb C^*: \sum_{m\in \Delta\cap M} t^{h(m)}f_m(x) = 0\} $$
where $f_m$, for each $m\in \Delta\cap M$ is the $\mathbb T^{d+1}$-invariant section associated to $m$. As $t\rightarrow 0$, the fibers $X_t = \{x\in Y:(x,t)\in X\}$ of this family degenerate to the toric boundary in $Y$. When $Y$ is smooth, the differential geometric aspects of this convergence is the subject of the well-known SYZ conjecture in mirror symmetry \cite{SYZ}. Equipping each such $X_t$ with the Calabi-Yau structure determined by the canonical polarization of $Y$, the metric SYZ conjecture states that for small $t$ a large part of $X_t$ admits a special Lagrangian $\mathbb T^d$ fibration. The base of this fibration is expected to be naturally identified with the boundary of the polytope
$$ \Delta_h^\vee := \{n\in N_\mathbb R: \langle m,n \rangle \leq h(m), \text{ for all } m\in \Delta\cap M \}. $$

Special Lagrangian $\mathbb T^d$-fibrations are intimately related to Monge-Ampère equations by ideas that go back at least to \cite{Hitchin}. Among other things, these ideas form the basis of the Kontsevich-Soibelman conjecture regarding Gromov-Hausdorff limits of $X_t$ (see \cite{GW, KS01,L}). 
More recently, results by Y. Li \cite{LiFermat,LiSYZ} puts the Monge-Ampère equations at the heart of the SYZ conjecture. Broadly speaking, \cite{LiFermat,LiSYZ} reduces the metric SYZ conjecture to structural properties for solutions to Monge-Ampère equations. In the general case, this pertains to the non-Archimedean Monge-Ampère equation and the structural properties sought is a subtle regularity condition.
In the toric case, the Monge-Ampère equation in question is the classical real Monge-Ampère equation on the open faces of $\Delta^\vee_h$, and the structural property sought is an extension property to $N_\mathbb R$. 
More precisely, let $\mathcal P(\Delta)$ be the space of convex function $\Psi$ on $N_\mathbb R$ such that 
$$\sup_{n\in N_\mathbb R} |\Psi(n)-\sup_{m\in \Delta\cap M} \langle m,n \rangle|<\infty.$$ 
Note that $\mathcal P(\Delta)$ by classical toric geometry corresponds to the torus-invariant continuous semi-positive metrics on the anti-canonical line bundle over $Y$ (see for example \cite{BB}). We will use $A$ and $B$ to denote the boundaries of $\Delta$ and $\Delta^\vee_h$ and $A^\circ$ and $B^\circ$ to denote the union of the relative interiors of the facets of  of $\Delta$ and $\Delta^\vee$. The sets $A^\circ$ and $B^\circ$ inherit \emph{tropical affine structures} from $M_\mathbb R$ and $N_\mathbb R$ (see Section~\ref{sec:AffineStructure}). These tropical affine structures induce canonical integral ''Lebesgue type'' measures $\mu_M, \nu_N$ on $A^\circ$ and $B^\circ$. If $\tau$ is a facet of $\Delta^\vee_h$ with relative interior $\tau^\circ$ and $\Psi\in \mathcal P(\Delta)$, then $\Psi|_{\tau^\circ}$ is convex with respect to this tropical affine structure and we get a well-defined Monge-Ampère measure $\MA(\Phi|_{\tau^\circ})$ on $\tau^\circ$. 

To establish the metric SYZ conjecture for $X$ using the ambient toric variety and the approach in \cite{LiFermat,LiSYZ} one needs a solution $\Psi\in \mathcal P(\Delta)$ to the Monge-Ampère equation
\begin{equation} \label{eq:MAIntLeb} \MA\left(\Psi|_{B^\circ}\right) = c\nu_N \end{equation}
where $c=\mu_M(A)/\nu_N(B)$ is a constant determined by $\Delta$ and $h$. Our main theorem provides a necessary and sufficient condition for existence of such solutions. The necessary and sufficient condition applies to data $(\Delta,h,\nu)$ where $\nu$ is a general positive measure of total mass $\mu_M(A)$ replacing the right hand side of \eqref{eq:MAIntLeb} and the condition is formulated in terms of \emph{optimal transport plans} from $\mu_M$ to $\nu$ with respect to the cost function $-\langle \cdot,\cdot\rangle$ on $A\times B$, i.e. couplings $\gamma$ of $\mu_M$ and $\nu$ minimizing the quantity
\begin{equation}
    \label{eq:OTCost}
    C(\gamma):= -\int_{A \times B} \langle m,n \rangle \gamma
\end{equation}
(see Section~\ref{sec:OptimalTransport}). 
For $m\in \Delta\cap M$, let $\tau_m$ denote the face of $\Delta^\vee_h$ dual to $m$, i.e. 
$$ \tau_m = \{n\in \Delta^\vee_h, \langle m,n \rangle = h(m) \}, $$ 
and $\St(m)$ be the closed star of $m$, i.e. the union of all closed faces of $A$ containing $m$. 

\begin{definition}\label{def:Stability}
We will say that $(\Delta,h,\nu)$ is \emph{stable} if there exists an optimal transport plan from $\mu_M$ to $\nu$ which is supported on 
\begin{equation} \label{eq:WeaklyStableSet} \cup_{m\in A\cap M} \left(\St(m) \times \tau_m\right).\end{equation}
\end{definition}

\begin{theorem}\label{thm:MAInt}
Let $\nu$ be a positive measure on $B^\circ$ of total mass $\mu_M(A)$. Then there is a function $\Psi\in \mathcal P(\Delta)$ satisfying 
\begin{equation} \label{eq:MAInt} \MA\left(\Psi|_{B^\circ}\right) = \nu \end{equation}
if and only if $(\Delta,h,\nu)$ is stable.  
Moreover, if $\Psi,\Psi'\in \mathcal P(\Delta)$ both satisfy \eqref{eq:MAInt}, then $\Psi-\Psi'$ is constant. 
\end{theorem}
\begin{remark}\label{rem:WeakSolution}
    In examples with abundant discrete symmetries, for example the standard unit simplex and the unit cube, we show that the solution is smooth on the open facets (see Theorem~\ref{thm:SmoothSolutions}, Lemma~\ref{lem:Symmetric} and Lemma~\ref{lem:SymmetricCube}). However, in general
    Equation~\ref{eq:MAIntLeb} and Equation~\ref{eq:MAInt} should be interpreted in a weak sense. More precisely, if $\beta:\tau^\circ\rightarrow \mathbb R^d$ is a coordinate function compatible with the tropical affine structure on the interior of a facet $\tau^\circ$, then $\Psi$ satisfies \eqref{eq:MAInt} if 
    $$ \partial(\Psi\circ\beta^{-1})(\beta(E)) = \nu(E) $$
    for any measurable set $E\subset \tau^\circ$, where $\partial$ denotes the (multi-valued) gradient. 
\end{remark}
\begin{remark}
    We will say that $(\Delta,h)$ is stable if $(\Delta,h,\nu_N)$ is stable and we will say that $\Delta$ is stable if $(\Delta,h_0)$ is stable, where $h_0$ is the trivial height function given by $h(0)=0$ and $h(m)=1$ for $m\in \Delta\cap M\setminus \{0\}$. 
\end{remark}
When $\Delta$ is the standard unit simplex (i.e. $Y=\mathbb P^{d+1}$), $h=h_0$ and $\nu$ is invariant under permutations of the vertices of $\Delta^\vee_h$, \eqref{eq:MAInt} was solved in \cite{HJMM}. If, in addition, $\nu$ is concentrated on $B^\circ$, $(\Delta,h_0,\nu)$, can be shown to be stable. 

At least when $Y$ is smooth and $h$ is trivial,
%(in particular, when $\Delta^\vee_h$ has the same number of facets as $\Delta^\vee$) 
existence of a solution $\Psi\in \mathcal P(\Delta)$ to \eqref{eq:MAIntLeb} implies the metric SYZ conjecture for $X$. Originally, this approach was deployed in \cite{LiFermat} to prove the metric SYZ-conjecture for the Fermat family (i.e. $Y=\mathbb P^n$ and $h=h_0$). In \cite{HJMM}, it was explained how to use a solution to \eqref{eq:MAInt} to directly verify the condition in \cite{LiSYZ}. This generalised the results in \cite{LiFermat} to a larger class of (possibly non-symmetric) families in $\mathbb P^{d+1}$. 
An independent generalisation was also achieved in \cite{PS}. Moreover, families of hypersurfaces in certain more general toric Fano manifolds was considered in \cite{LiToric}. Theorem~\ref{thm:MAInt} have the following corollary 
(see \cite{LiToric}, Section~3, for details on how the corollary follows from Theorem~\ref{thm:MAInt} above):
%(see Appendix~\ref{sec:SYZ} for details on how the corollary follows from Theorem~\ref{thm:MAInt} above):
\begin{corollary}\label{cor:SYZ}
    Assume $Y$ is smooth, $h=h_0$ and $\Delta$ is stable. Then there is for each $\delta>0$, some $\epsilon=\epsilon_\delta>0$ such that $X_t$ admits a special Lagrangian $\mathbb T^d$ fibration on a subset of normalized Calabi-Yau volume $1-\delta$ whenever $|t|<\epsilon$.
\end{corollary}
\begin{remark}
    In general, the subset admitting a special Lagrangian torus fibration in Corollary~\ref{cor:SYZ} is not explicit. However, for the hypersurfaces in $\mathbb P^{d+1}$ and $(\mathbb P^1)^{d+1}$, the regularity result for the standard unit simplex and the unit cube mentioned above can be used to extract more precise information (see Corollary~\ref{cor:SmoothSYZ}) below. 
\end{remark}
One of the key ideas in \cite{HJMM} is a new variational principle for \eqref{eq:MAInt}. This variational principle will play a key role here. A striking feature of \cite{LiToric} is that it doesn't rely on symmetries of $\Delta$, unlike \cite{LiFermat, HJMM, PS}. However, \cite{LiToric} relies on a condition on the vertices of $\Delta$ and $\Delta^\vee$, which unfortunately seems rather restrictive (see Section~\ref{sec:Computer Aided}). 

Examining the stability condition in Definition~\ref{def:Stability}, one finds that if $n$ is a point in a facet $\tau$ of $B$ then the cost function $-\langle \cdot,n \rangle$ achieves its minimum at $m_\tau$. It is thus reasonable to expect that as long as $\Delta$ and $h$ are not too wild, the optimal transport plan will be supported on \eqref{eq:WeaklyStableSet} and $(\Delta,h,\nu_N)$ is stable. In Section~\ref{sec:Examples} we explain how symmetries in $\Delta$ and $h$ can be used to prove stability of $(\Delta,h,\nu_N)$. 
%In fact, for the integral measures on the standard unit simplex and the unit cube, we prove existence of a solution which is \emph{smooth} everywhere but on a codimension 2 subset (Theorem~\ref{thm:SmoothSolutions}, Lemma~\ref{lem:Symmetric} and Lemma~\ref{lem:SymmetricCube}).
%, reproducing results on the SYZ conjecture from \cite{HJMM,PS,LiToric}.
However, somewhat surprisingly, there is a large number of cases when stability fails and \eqref{eq:MAIntLeb} doesn't admit a solution, even when $\nu=\nu_N$. We will briefly summarize these findings here. More details are provided in Section~\ref{sec:Examples}. 
Note that $\Delta^\vee_{h_0}$ is the usual dual $\Delta^\vee$ of $\Delta$. 
\begin{itemize}
    \item Let $\Delta$ be the standard 2-simplex and hence $Y=\mathbb P^2$. There is a height function $h$ such that \eqref{eq:MAIntLeb} does not admit a solution (Example~\ref{ex: unstable 1d nontrivial height} in Section~\ref{sec:Examples}).
    \item Let $h=h_0$ be the trivial height function. Then, among the 4319 reflexive polytopes of dimension three, at least 1542 do not admit solutions to \eqref{eq:MAIntLeb} (Table~\ref{tab:d2} in Section~\ref{sec:Examples}). 
    \item Let $h=h_0$ be the trivial height function. Among 4319 reflexive polytopes of dimension three, 145 admit solutions but are very close to being unstable (see Definition~\ref{def: strict strictural semistability}). As explained in the next two paragraphs, these examples express unexpected behaviours with respect to the singular sets of the associated tropical affine structures.  
\end{itemize}

In general, equation~\eqref{eq:MAInt} can be relaxed by considering general polarizations of $Y$. The first bullet above is interesting since it describes a case when this relaxation gives no additional freedom. We stress, however, that solvability of \eqref{eq:MAIntLeb} is only a sufficient condition for the metric SYZ-conjecture to hold for $X$. What the first bullet point tells us is that for an approach similar to \cite{LiFermat} to work for the family defined by Example~\ref{ex: unstable 1d nontrivial height}, another ambient toric variety than $\mathbb P^2$ has to be considered. 

It is generally expected that the solution to the Monge-Ampère equation \eqref{eq:MAInt} can be extended to a set which is larger than $B^\circ$, in particular it should extend to a set $B\setminus \Sigma$ where $\Sigma$ has codimension two. This will necessarily involve a choice of tropical affine structure extending the tropical affine structure on $B^\circ$. A priori, there does not seem to be a canonical choice of such a structure. In particular, extending the tropical affine structure involves a seemingly arbitrary choice of location for its singular set. In Section~\ref{sec:OutsideOpenFaces} we address this by arguing that the location of the singular set needs to be chosen to suit the Monge-Ampère equation and thus plays the role of a free boundary in terms of PDE theory. We give a precise definition of the extension of the tropical affine structure (Definition~\ref{def:Usigma}) and show that if \eqref{eq:MAInt} admits a solution, then it extends to a solution on the regular part of this tropical affine structure (Theorem~\ref{thm:MA}). In two examples, the standard unit simplex and the unit cube, we show in addition that the singular set of the induced tropical affine structure $\Sigma_\Psi$ is of codimension 2 and that the solution (or, more precisely, the potentials in \eqref{def:PDEproblem}) are smooth on $B\setminus \Sigma_\Psi$ (Theorem~\ref{thm:SmoothSolutions}, Lemma~\ref{lem:Symmetric} and Lemma~\ref{lem:SymmetricCube}). 

For general data $\Delta,h,\nu$, the exact definition of $\Sigma$ is somewhat technical, due to the lack of a regularity theory for $\Psi$. However, $\Psi$ defines a multivalued map $\partial^c\Psi:B\rightarrow A$ (the $c$-gradient) and assuming this map is a homeomorphism, the definition of $\Sigma$ reduces to 
\begin{equation} \Sigma = \Sigma_\Psi = B_{d-1}\cap (\partial^c\Psi)^{-1}(A_{d-1}) \label{eq:Sigma when homeomorphism} \end{equation}
where $A_{d-1} = A\setminus A^\circ$ and $B_{d-1}=B\setminus B^\circ$ are the $(d-1)$-skeletons of $A$ and $B$. If we use $\partial^c\Psi$ to identify $A$ and $B$ this fits well into the point of view put fourth by many authors that the singular set should be the intersection of the codimension 1 skeletons of two dual polyhedral structures on the unit sphere (see for example \cite{KS06}, end of Section~2.2, and \cite{LiToric}, Section~2.10) 
and assuming $B_{d-1}$ and $(\partial^c\Psi)^{-1}(A_{d-1})$ intersect transversely this gives a set $\Sigma_\Psi$ of codimension two, as is generally expected in the SYZ conjecture. 

Three interesting questions related to this are 
\begin{itemize}
\item Size: How big is $\Sigma_\Psi$? In particular, are there suitable conditions under which $\Sigma_\Psi$ is of codimension 2?
\item Minimality: Is $\Sigma_\Psi$ minimal, or does there exist a closed proper subset $\Sigma'$ of $\Sigma_\Psi$ such that the tropical affine structure extends to $B\setminus \Sigma'$ and the Monge-Ampère equation is satisfied on this larger set
\item Uniqueness: Assuming $\Sigma_\Psi$ is minimal, does there exist a closed set $\Sigma'\subset B$, not containing $\Sigma_\Psi$, such that the tropical affine structure on $B^\circ$ extends to $B\setminus \Sigma'$ and the Monge-Ampère equation is satisfied on this larger set
\end{itemize}
When considering the second and third bullet points above, we will assume $\Sigma_\Psi$ and $\Sigma'$ are the singular sets of tropical affine structures whose coordinate functions are defined by facets of $\Delta$ (see Section~\ref{sec:AffineStructure}). Varying the affine structure will come down to changing the domain of these coordinate functions. Under this constraint, Section~\ref{sec:OutsideOpenFaces} give strong evidence for the third bullet point. More precisely, assuming $\Sigma_\Psi$ is minimal we prove that for any $\Sigma'$ not containing $\Sigma_\Psi$, the set $B\setminus \Sigma'$ will necessarily contain points where the solution is non-differentiable  (see Lemma~\ref{lem:SingularOutsideCharts}). This answers the third bullet at least for $d=2$, where solutions are smooth. On the other hand, in Section~\ref{sec:Examples} we give a large class of examples such that \eqref{eq:MAInt} admits a solution but $\partial^c\Psi$ maps entire $(d-1)$-dimensional faces of $B_{d-1}$ into $A_{d-1}$. In particular, this means $\Sigma_\Psi$, as defined above, is of codimension 1. In fact, in these examples $B\setminus \Sigma_\Psi$ is not connected. These examples are stable but very close to being unstable (see Definition~\ref{def: strict strictural semistability}) and we argue that the second bullet point above is likely failing for these examples. More precisely, it seems that $\Sigma_\Psi$ can be replaced by a smaller subset $\Sigma'$ which is of codimension 2 (see Example~\ref{ex: admissable and semistable} and Remark~\ref{rem:Removable Singularity}). Curiously, $\Sigma'$ contains $(d-2)$-dimensional faces of $B$. Together with \cite{HJMM}, where the singular set induces a barycentric subdivision of the $(d-1)$-dimensional faces, this suggests a heuristic picture where the singular set is pushed towards the boundary of some $(d-1)$-dimensional faces when $\nu$ (or $h$, if we let $h$ take values in $\mathbb R$) moves towards data which is not stable. 

The exact subset of $X_t$ admitting a special Lagrangian torus fibration in Corollary~\ref{cor:SYZ} is not explicit. Loosely speaking, it lives over the smooth locus in the open facets of the solution $\Psi$. Since the singular set of local solutions to real Monge-Ampère equations can be quite wild there is not much hope of getting a more explicit description without regularity results for $\Psi$. Moreover, since the base is contained in the open facets of $\Delta^\vee$, the subset furnished by Corollary~\ref{cor:SYZ} is definitely not connected. On the other hand, the regularity result for the standard unit simplex and the unit cube mentioned above implies a more precise version of Corollary~\ref{cor:SYZ}. To state it, let $T_\bC\subset Y$ be the complex $(d+1)$-torus and $\Log_s:T_\bC\rightarrow N_\bR$ the map defined by sending $x\in T_\bC$ to the unique $\Log_s(x)=n\in N_\bR$ such that 
$$ \langle m,\Log_s(x) \rangle = \frac{1}{s}\log|f_m(x)| $$
for all $m\in M$. If we fix generators $m_0,\ldots,m_d$ of $M$ these determines coordinates $(\langle m_0,\cdot \rangle,\ldots,\langle m_d,\cdot \rangle)$ on $N_\bR$ and 
$$ (z_0,\ldots,z_d) = (f_{m_0}, \ldots, f_{m_d}) $$
on $T_\bC$ and $\Log_s$ takes the form
$$ \Log_s(z_0,\ldots,z_d) = \frac{1}{s}(\log|z_0|,\ldots,\log|z_d|). $$
Let $\Delta_{simplex}$ and $\Delta_{cube}$ be the standard unit simplex and the unit cube in $M_\bR$, respectively. For each facet $\sigma$ of $\Delta_{simplex}$ and $\Delta_{cube}$, let $\SmSt(n_\sigma)^\circ$ be the open star of $n_\sigma$ in the barycentric subdivison of $\Delta_{simplex}$ or $\Delta_{cube}$. 
\begin{corollary}\label{cor:SmoothSYZ}
    Let $K_{faces}$ be a compact subset of $B^\circ$ and $\tilde U_{faces}=R_{\geq 0}K_{faces}$ be the cone generated by $K_{faces}$. For each facet $\sigma$ of $\Delta_{sim}$ (or $\Delta_{cube}$), let $K_\sigma$ be a compact subset of $\SmSt(n_\sigma)^\circ$. Let
    %Let $$\tilde U_{faces} = \cup_\tau \bR_{\geq 0}K_\tau$$ 
    %be the cone generated by $\cup_\tau K_\tau$ and 
    $$ \tilde U_{stars} = \cup_\tau( [0,1]\cdot K_\sigma + \bR_{\geq 0}n_\sigma). $$
    Then for small $t$, $X_t$ admits a special Lagrangian torus fibration on 
    $$ \Log_t^{-1}\left(\tilde U_{faces} \cup \tilde U_{stars}\right). $$
\end{corollary}
We will explain in the appendix how Corollary~\ref{cor:SmoothSYZ} follows from Theorem~\ref{thm:MA} and the regularity results in Theorem~\ref{thm:SmoothSolutions}, Lemma~\ref{lem:Symmetric} and Lemma~\ref{lem:SymmetricCube} for the standard unit simplex and unit cube. Briefly put, the argument in \cite{LiSYZ}, which forms the base of the arguments in \cite{HJMM} and the appendix of \cite{LiToric}, only pertains to the open faces of $\Delta^\vee$. To address (parts of) the lower dimensional faces we use the setup of \cite{LiFermat}. The argument relies on the uniform Skoda estimate in \cite{LiSkoda} and the application of $L^1$-stability in \cite{LiSYZ}. However it bypasses all non-Archimedean geometry (including the solution to the non-Archimedean Calabi-Yau problem, its Fubini-Study approximations of and their regularizations). Compared to \cite{LiFermat} it is also somewhat simpler since it relies on existence of solutions from Theorem~\ref{thm:MA} instead of constructing solutions as limits of (double Legendre regularizations) of averages of Calabi-Yau potentials. Consequently, the exposition might be of independent interest.

\subsection*{Outline}
In Section~\ref{sec:Reflexive Polytopes} we recall some basic properties of reflexive polytopes, most importantly a projection property for $\Delta$ and $\Delta_h$. In Section~\ref{sec:AffineStructure} we explain how ideas from \cite{GS,HZ} give a natural tropical affine structures on large subsets of $A$ and $B$. In Section~\ref{sec:Pairing} we show how these affine structures respect the pairing between $M_\mathbb R$ and $N_\mathbb R$ and in Section~\ref{sec:cGradient} we explain how this can be used to control the $c$-gradient of a function $\Psi\in \mathcal P(\Delta)$. Section~\ref{sec:OptimalTransport} recalls some facts from optimal transport theory, in particular the principle that a transport plan is optimal if and only if it is supported on the $c$-gradient of a $c$-convex function. Section~\ref{sec:MA} puts these ideas together to prove Theorem~\ref{thm:MAInt}. In Section~\ref{sec:OutsideOpenFaces} we define $\Sigma_\Psi$ and show how to extend the affine structure on $B^\circ$ to $B\setminus \Psi$ and prove that the solution to \eqref{eq:MAInt} satisfies a Monge-Ampère equation on this larger set (see Definition~\ref{def:PDEproblem}). We also address the third bullet regarding $\Sigma_\Psi$ above. 

In Section~\ref{sec:Examples} we present examples that highlight important features of our results. In particular, Section~\ref{sec:Examples Existence} gives positive existence results using symmetry, Section~\ref{sec:Examples non-Existence} gives examples when \eqref{eq:MAInt} fails to have a solution, Section~\ref{sec:Examples Anomalous} presents examples such that the extension of the tropical affine structure has an anomalous singular set and Section~\ref{sec:Computer Aided} presents quantitative results which show that the properties exhibited in these examples are frequently occurring in reflexive polytopes. 

\subsection*{Acknowledgement}
The authors would like to thank Yang Li for helpful comments on a draft of the paper. The second author would like to thank Mattias Jonsson, Enrica Mazzon and Nick McCleerey for many fruitful discussions on the subject. The second author was supported by the Knut and Alice Wallenberg Foundation, grant 2018-0357.

\section{Reflexive Polytopes}\label{sec:Reflexive Polytopes}
A reflexive polytope is a lattice polytope $\Delta$ whose dual $\Delta^\vee$ is also a lattice polytope. Equivalently, $\Delta$ can be written as 
$$ \{ m: \langle m,n_i \rangle \leq 1 \} $$
for a number of lattice vectors $n_1,\ldots,n_k$ (the vertices of $\Delta^\vee)$. The definition of reflexive polytopes, along with the idea to study them in mirror symmetry, goes back to \cite{Batyrev}. The condition puts strong restrictions on $\Delta$. Since $n_i$ is a lattice vector, we get that $\langle m,n_i \rangle\in \mathbb Z$ for each $m\in M$, consequently, if $m\in \Delta\cap M$ and $\langle m,n_i \rangle>0$ for some $i$, then $\langle m,n_i \rangle =1$. From this one may conclude the following:
\begin{itemize}
\item If $\sigma$ is a facet of $\Delta$, then there are no lattice points strictly between the affine subspace spanned by $\sigma$ and the parallel subspace passing through the origin.
\item Each $n_i$ is primitive, i.e. there is no $q>1$ such that $n_i/q\in N$. 
\end{itemize}
With a little bit more effort, we also get the follow two statements regarding projections of $\Delta$ and $\Delta^\vee_h$.  
\begin{lemma}[See also \cite{Nill}, Prop 2.2 1]
\label{lem:ProjectionDelta}
Let $m_0\in (\Delta\cap M)\setminus \{0\}$ and $\pi_{m_0}:M_\mathbb R\rightarrow M_\mathbb R/\mathbb Rm_0$ be the projection map. Then $\pi_{m_0}(\St(m_0))=\pi_{m_0}(\Delta)$.
\end{lemma}
\begin{proof}
Let 
$$ V_{m_0} = \{m\in \Delta: m+\epsilon m_0\notin \Delta \text{ for all } \epsilon>0 \} $$
And note that the image of $\Delta$ under the projection above is contained in the image of $V_{m_0}$. We claim that $V_{m_0}=\St(m_0)$. To see this, let first $m\in \St(m_0)$. It follows that $m$ and $m_0$ lie on a common facet, i.e. $\langle m,n_i \rangle = \langle m_0,n_i \rangle = 1$ for some $n$. We get $(m+\epsilon m_0)n_i = 1+\epsilon$, hence $m+\epsilon m_0\notin \Delta$ for any $\epsilon>0$. Conversely, assume $m\in V_{m_0}$. Then there is $i$ such that
$$ 1 < \langle m+\epsilon m_0, n_i\rangle = \langle m, n_i\rangle + \epsilon\langle m_0, n_i\rangle $$ 
for all $\epsilon$. 
It follows that $\langle m,n_i\rangle = 1$ and $\langle m_0,n_i \rangle> 0$, and by the observation above that $\langle m_0,n_i \rangle =1$. Consequently, $m_0$ and $m$ lie on the facet of $\Delta$ defined by $n_i$ and $m\in \St(m_0)$. 
\end{proof}
Closely related to the previous lemma, we have the following lemma regarding projections of $\Delta^\vee_h$.
\begin{lemma}\label{lem:ProjectionDeltaDual}
Let $\sigma$ be a facet of $\Delta$ and $n_\sigma\in N$ be its primitive outward normal. Let $\pi_{n_\sigma}: N_\mathbb R \rightarrow N_\mathbb R/\mathbb R n_\sigma$ be the projection map. Then 
$$ \pi_{n_\sigma}(\cup_{m\in \sigma\cap M} \tau_m) = \pi_{n_\sigma}(\Delta^\vee_h). $$
\end{lemma}
\begin{proof}
As in the proof of Lemma~\ref{lem:ProjectionDelta}, let
$$ V_{n_\sigma} = \{n\in \Delta^\vee_h: n+\epsilon n_\sigma\notin \Delta^\vee_h \text{ for all } \epsilon>0 \} $$
and note that $\pi_{n:\sigma}(V_{n_\sigma})=\pi_{n:\sigma}(\Delta^\vee_h)$. A similar computation as in the proof of Lemma~\ref{lem:ProjectionDelta} gives that $V_m$ is the union of those facets $\tau$ of $\Delta^\vee_h$ whose primitive outward normal $m_\tau$ satisfies $\langle m_\tau,n_\sigma \rangle > 0$, and hence $\langle m_\tau,n_\sigma \rangle = 1$, i.e. the set of facets of $\Delta^\vee_h$ whose primitive outward normal lie in $\sigma$. Consequently, $V_{n_\sigma}=\cup_{m\in \sigma\cap M}\tau_m$.
\end{proof}

\section{Tropical affine structures}\label{sec:AffineStructure}
Let $M$ be a topological manifold of dimension $d$. An \emph{affine structure} on $M$ is a special atlas $\{U_i,\beta_i:U_i\rightarrow \mathbb R^d\}_{i=1}^k$ such that for all $i,j$, the transition function $\beta_i\circ \beta_j|_{U_i\cap U_j}^{-1}$ is affine, i.e. on the form $y\rightarrow A_{ij}y+b_{ij}$ where $A_{ij}\in \GL_d(\mathbb R)$ and $b_{ij}\in \mathbb R^d$. If, in addition, $A_{ij}\in \SL_d(\mathbb Z)$ for all $i,j$, then $\{U_i,\beta_i:U_i\rightarrow \mathbb R^d\}_{i=1}^k$ defines a \emph{tropical affine structure} on $M$. Note that an affine structure, since its transition functions are smooth, determines a smooth structure on $M$ and, equivalently to the definition above, a tropical affine structure is defined by an affine structure and a lattice in the tangent space of $M$. 

Let $\sigma$ be a facet of $\Delta$ with primitive outward normal $n_\sigma$ and $\tau$ a facet of $\Delta^\vee_h$ primitive outward normal $m_\tau$. As open subsets of affine subspaces of $M_\mathbb R$ and $N_\mathbb R$ respectively, $\sigma^\circ$ and $\tau^\circ$ inherit affine structures. Moreover, the sublattices $n_{\sigma}^\perp\cap M=\{m\in M: \langle m,n_\sigma \rangle =0\}$ and $m_{\tau}^\perp\cap N=\{n\in N: \langle m_\tau,n \rangle = 0\}$ define lattices on the tangent spaces of $\sigma^\circ$ and $\tau^\circ$, respectively, and hence tropical affine structures on $\sigma^\circ$ and $\tau^\circ$.

Let $m_1,\ldots,m_d$ be generators of $n_\sigma\cap M$ and
$ \beta_\sigma : \Delta^\vee_h \rightarrow \mathbb R^d $
be the map given by
$$ \beta_\sigma = (m_1(n),\ldots,m_d(n)). $$
Note that $\beta_\sigma$ depend on the choice of generators $m_1,\ldots,m_d$, but only up to (left) composition with a map in $\SL_d(\mathbb Z)$. More precisely, $m_1,\ldots,m_d$ define an identification $\iota$ of $N_\mathbb R/\mathbb Rn_\sigma$ and $\mathbb R^d$ and $\beta_\sigma$ is the composition of $\iota$ and the projection map $\pi_\sigma:N_\mathbb R\rightarrow N_\mathbb R/\mathbb Rn_\sigma$. We will see in the next subsection that it is sometimes advantageous to pick certain sets of generators. Since $\langle m,n_\sigma\rangle=1>0$ for all $m\in \sigma$ we get that $\beta_\sigma$ is one-to-one on every facet $\tau'$ of $\Delta$ whose primitive outward normal lies in $\sigma$. It follows that $\beta_\sigma$ is one to one on $\cup_{m\in \sigma\cap M} \tau_m$. If $h(m) = 1$ for all $m\in A\cap M$ then $n_\sigma\in \Delta^\vee_h=\Delta^\vee$ and $\cup_{m \in \sigma\cap M} \tau_m = \mathrm{St}(n_\sigma)$. 

In a similar manner, picking a set of generators $n_1,\ldots,n_d$ of $m_\tau^\perp\cap N$ defines a map
$ \alpha_\tau : \Delta \rightarrow \mathbb R^d $
given by
$$ \alpha_\tau(m) = (\langle m, n_1\rangle,\ldots,\langle m,n_d\rangle) $$
which is one-to-one on $\cup_{\sigma: m_\tau\in \sigma} \sigma = \St(m_\tau)$.

\begin{lemma}\label{lem:ias}
$\{\tau^\circ,\beta_{\sigma}|_{\tau^\circ}\}_{(\tau,\sigma)\in B_d,A_d: m_\tau\in \sigma}$ is an atlas for the tropical affine structure on $B^\circ$ and $\{\sigma^\circ,\alpha_{\tau}|_{\sigma^\circ}\}_{(\tau,\sigma)\in B_d,A_d: m_\tau\in \sigma}$ is an atlas for the tropical affine structure on $A^\circ$. 
\end{lemma}
\begin{proof}
Let $\tau$ and $\sigma$ satisfy $m_\tau\in \sigma$. First of all, since $\langle m_\tau,n_\sigma \rangle = 1 \not= 0$ we get that $\beta_\sigma$ define a non-degenerate affine map from the affine subspace spanned by $\tau$ to $\mathbb R^d$. It follows that $\beta_\sigma$ respect the affine structure on $\tau^\circ$. 

To prove the first part of the lemma, it thus suffices to show that $\beta_\sigma$ respect the lattice structure on the tangent space of $\tau$, in other words that if $n\in \affspan(\tau)$ then $\beta_\sigma(n)\in \mathbb Z^d$ if and only if $n\in N$. First of all, if $n\in N$ then $\beta_\sigma(n)\in \mathbb Z^d$ since the generators $m_1,\ldots,m_d$ of $n_\sigma^\perp$ used in the definition of $\beta_\sigma$ lie in $M$.  
For the converse, let $e_0,\ldots,e_d$ be generators of $M$. By Lemma~\ref{lem:Wedge} below, we have
$$ m\wedge m_1 \wedge  \ldots \wedge m_d = \langle m,n_\sigma \rangle e_0\wedge \ldots \wedge e_d  = e_0\wedge \ldots \wedge e_d $$
hence $m,m_1,\ldots,m_d$ generate $M$. Since $N$ is dual to $M$, it follows that if $n\in \affspan(\tau)$ (in other words $\langle m_\tau,n \rangle = h(m_\tau)\in \mathbb Z$) and $\beta_\sigma(n) = (\langle m_1,n\rangle,\ldots,\langle m_d,n \rangle)\in \mathbb Z^d$, then $n\in N$. A similar argument proves the second part of the lemma.

\end{proof}
\begin{lemma}\label{lem:Wedge}
Let $e_0,\ldots,e_d$ be a set of generators of $M$, $n\in N$ be primitive and $m_1,\ldots,m_d$ be a set of generators of $n^\perp \cap M$. Then 
\begin{equation} \label{eq:Wedge} m\wedge m_1\wedge\ldots\wedge m_d = \langle m,n \rangle e_0\wedge \ldots \wedge e_d \end{equation}
for all $m\in M_\mathbb R$.
\end{lemma}
\begin{proof}
First of all, it suffices to prove the lemma for some set of generators $e_0,\ldots,e_d$ and some set of generators $m_1,\ldots,m_d$ since changing these does not affect the right or left hand side of \eqref{eq:Wedge}.

Since $n$ is primitive, we may pick a set of generators $f_0,\ldots,f_d$ of $N$ such that $n=f_0$. Let $e_0,\ldots,e_d$ be the dual set of generators for $M$. Then $e_1,\ldots,e_d$ generate $n_\sigma^\perp \cap M$ so we can let $m_i=e_i$ for $i=1,\ldots,d$. Writing $m=\sum_{i=0}^d \alpha_i e_i$, we get 
\begin{eqnarray*}
m\wedge m_1\wedge\ldots\wedge m_d & = & \left(\sum_{i=0}^d \alpha_i e_i\right)\wedge e_1\wedge\ldots\wedge e_d \\
& = & a_0 e_0\wedge e_1\wedge\ldots\wedge e_d \\
& = & \langle m,f_0 \rangle e_0\wedge e_1\wedge\ldots\wedge e_d \\
& = & \langle m,n \rangle e_0\wedge e_1\wedge\ldots\wedge e_d.
\end{eqnarray*}
\end{proof}

\begin{lemma}
Let $\sigma$ be a facet of $\Delta$ and $\tau$ a facet of $\Delta^\vee_h$ and $\alpha_\tau$ and $\beta_\sigma$ be defined as above. Then
$$ \alpha_\tau(\Delta) = \alpha_\tau(\St(m_\tau)) \text{ and } \beta_\sigma(\Delta^\vee_h) = \beta_\sigma(\cup_{m\in \sigma\cap M} \tau_m). $$
\end{lemma}
\begin{proof}
As noted above, $\beta_\sigma=\iota\circ\pi_\sigma$ where $\pi_\sigma:N_\mathbb R\rightarrow N_\mathbb R/\mathbb R n_\sigma$ is the projection map and $\iota$ is an identification of $N_\mathbb R/\mathbb R n_\sigma$ and $\mathbb R^d$. Given this, the second part of the lemma follow from Lemma~\ref{lem:ProjectionDeltaDual}. The first part of the lemma follows from a similar application of Lemma~\ref{lem:ProjectionDelta}.
\end{proof}

Note that if $\tau$ is a facet of $\Delta^\vee_h$, then $\Psi|_\tau$, and hence $\Psi|_\tau\circ \beta_\sigma^{-1}$ for any $\sigma$ such that $m_\tau\in \sigma$, is convex since it is the restriction of a convex function to an affine subspace. It turns out that, at least after modifying $\Phi$ by a linear function, a similar statement holds for the larger sets where $\beta_\sigma$ is one-to-one.  
\begin{lemma}\label{lem:ConvexPotential}
Let $\Psi\in \mathcal P(\Delta)$, $\sigma$ be a facet of $\Delta$ and $m_0\in \sigma$. Assume $\beta$ is defined as above. Then $(\Psi-m_0)\circ \beta_\sigma^{-1}$ is convex on $\beta_\sigma(\cup_{m\in \sigma}\tau_m)$.
\end{lemma}
\begin{proof}
We may write $\Psi=\Phi^*$ for $\Phi=\Psi^*\in L^\infty(\Delta)$. This means
$$ (\Psi-m_0)\circ \beta^{-1}(y) = \sup_{m\in \Delta} \langle m-m_0,\beta^{-1}(y)\rangle-\Phi(m) $$
for all $y\in \beta(\cup_{m\in\sigma\cap M} \tau_m)$, hence it suffices to show that 
$$(m-m_0)\circ \beta_\sigma^{-1}(y) = \langle m-m_0,\beta^{-1}(y)\rangle$$ 
is convex on $\beta(\cup_{m\in\sigma\cap M} \tau_m)$ for any $m\in \Delta$. To do this, recall as above that $\beta_\sigma = \iota\circ \pi_\sigma$ where $\pi_\sigma:N_\mathbb R\rightarrow N_\mathbb R/\mathbb R n_\sigma$ is the projection map and $\iota$ is an identification of $N_\mathbb R/\mathbb R n_\sigma$ and $\mathbb R^d$. Fixing a hyperplane $H$ in $N_\mathbb R$ not containing $n_\sigma$ and identifying this in the natural way to $N_\mathbb R/\mathbb Rn_\sigma$ we get that $\beta^{-1} = l\circ \iota'$ where $\iota':\mathbb R^d\rightarrow H$ is an invertible linear map and $l:H\rightarrow \cup_{m\in \sigma\cap M} \tau_m$ is of the form $l(n) = n-n_\sigma h(n)$ for some piecewise affine $h:H\rightarrow \mathbb R$ whose non-differentiable locus is exactly the codimension 1 skeleton of $\cup_{m'\in \sigma\cap M} \tau_{m'}$. We claim that $h$ is convex. Assume for a contradiction that this is not true, hence there are $n_0,n_1\in H$ such that $(h(n_0)+h(n_1))/2 < h((n_0+n_1)/2)$. Since $n-n_\sigma h(n)\in \cup_{m\in\sigma\cap M} \tau_m$ and hence 
$$ \sup_{m\in \sigma\cap M} \langle m, n-n_\sigma h(n)\rangle = 1$$
for any $n\in H$ we get the contradiction
\begin{eqnarray*}  1 & = & \sup_{m\in \sigma} \langle m, (n_0+n_1)/2-n_\sigma h((n_0+n_1)/2) \rangle \nonumber \\ 
& < & \sup_{m\in \sigma} \frac{1}{2}\langle m,n_0-n_\sigma h(n_0) \rangle + \frac{1}{2}\langle m,n_1-n_\sigma h(n_1) \rangle \nonumber \\
& \leq & \sup_{m\in \sigma} \frac{1}{2}\langle m,n_0-n_\sigma h(n_0) \rangle + \sup_{m\in \sigma} \frac{1}{2}\langle m,n_1-n_\sigma h(n_1) \rangle \nonumber \\
& = & 1.
\end{eqnarray*}

Now, 
\begin{eqnarray*}
    \langle m-m_0,l(n) \rangle & = & \langle m-m_0, n-n_\sigma h(n) \rangle \\
    & = & \langle m-m_0, n \rangle + \langle m-m_0, n_\sigma \rangle h(n).
\end{eqnarray*}
The first term in this is linear in $n$ and the second term is convex in $n$ since 
$$\langle m-m_0, n_\sigma \rangle = \langle m, n_\sigma \rangle - \langle m_0,n_\sigma \rangle \leq 1-1 = 0.$$
It follows that $(m-m_0)\circ \beta_\sigma^{-1}(y)=\langle m-m_0,l(\iota'(y)) \rangle $ is convex in $y$.
\end{proof}

\section{Compatibility with pairings}\label{sec:Pairing}
Let $\sigma$ be a facet of $\Delta$ and $\tau$ a facet of $\Delta^\vee_h$. The definitions of $\alpha_\tau$ and $\beta_\sigma$ depend on a choice of generators for $m_\tau^\perp\cap N$ and a choice of generators for $n_\sigma^\perp\cap M$, respectively. In this section, we will explain that by choosing these sets of generators in a good way, we can get pairs $(\alpha_\tau,\beta_\sigma)$ which are compatible with the pairing on $M_\mathbb R\times N_\mathbb R$ in the following sense:
\begin{definition}
    Let $\sigma$ be a facet of $\Delta$ and $\tau$ a facet of $\Delta^\vee_h$. A pair 
    $\alpha_\tau: \Delta\rightarrow \mathbb R^d$ and $\beta_\sigma:\Delta^\vee_h \rightarrow \mathbb R^d$ are $(\sigma,\tau)$-compatible if 
    \begin{equation} \label{eq:CompatibleCharts} \langle m-m_\tau,n \rangle = \langle \alpha_\tau(m),\beta_\sigma(n) \rangle \end{equation}
    for all $m,n \in \sigma\times \Delta^\vee_h$. and 
    \begin{equation} \label{eq:CompatibleChartsDual} \langle m,n-h(m_\tau)n_\sigma \rangle = \langle \alpha_\tau(m),\beta_\sigma(n) \rangle \end{equation}
     for all $m,n \in \Delta\times \tau$.    
\end{definition}
\begin{lemma}\label{lem:CompatibleCharts}
Let $\sigma$ be a facet of $\Delta$ and $\tau$ be a facet of $\Delta^\vee_h$ such that $m_\tau\in \sigma$. 

Assume 
$\beta_\sigma $
is a map defined as in the previous section. Then there is a coordinate map $\alpha_\tau$ 
such that $(\alpha_\tau, \beta_\sigma)$ are $(\sigma,\tau)$-compatible. 

Similarly, if $\alpha_\tau$ is a map defined as in the previous section, then there is a coordinate map $\beta_\sigma$ such that $(\alpha_\tau, \beta_\sigma)$ are $(\sigma,\tau)$-compatible.
\end{lemma}
\begin{proof}
Let $m_1,\ldots,m_d$ be the generators of $n_\sigma^\perp\cap M$ that defines $\beta_\sigma$. As explained in the proof of Lemma~\ref{lem:ias}, $m_\tau,m_1,\ldots,m_d$ generate $M$. As $\langle m_i,n_\sigma \rangle=0$ for all $i$, $\langle m_\tau,n_\sigma \rangle = 1$ and $\langle m_\tau,n \rangle = 0$ for any $n\in m_\tau^\perp$ it follows that $m_\tau^\perp \cap N $ is dual to the lattice generated by $m_1,\ldots,m_d$. Consequently, we can pick generators $n_1,\ldots,n_d$ for $m_\tau^\perp \cap N$ which are dual to $m_1,\ldots,m_d$, i.e. $\langle m_i,n_j \rangle = 1$ if $i=j$ and $\langle m_i,n_j \rangle = 0$ if $i\not=j$. Writing an element $m\in \sigma$ as $m=m_\tau+\sum_{i=1}^d \alpha_i m_i$ and an element $n\in \Delta^\vee_h$ as $n = \beta_0n_\sigma + \sum_{j=1}^d \beta_j n_j$ for coefficients $\alpha_1,\ldots,\alpha_d,\beta_0,\ldots,\beta_d$ we get $\alpha_\tau(m) = (\alpha_1,\ldots,\alpha_d)$, $\beta_\sigma(n) = (\beta_1,\ldots,\beta_d)$ and 
\begin{eqnarray*}
\langle m-m_\tau,n \rangle & = & \langle \sum_{i=1}^d \alpha_i m_i, \beta_0n_\sigma + \sum_{j=1}^d \beta_j m_j \rangle \\
& = & \sum_{i,j=1}^d \alpha_i\beta_j \langle m_i,n_j \rangle \nonumber \\
& = & \sum_{i}^d \alpha_i\beta_i \nonumber \\
& = & \langle \alpha_\tau(m),\beta_\sigma(n)\rangle. \nonumber
\end{eqnarray*}
This proves \eqref{eq:CompatibleCharts}. A similar computation, writing an element $m\in \Delta$ as $\alpha_0m_\tau+\sum_{i=1}^d \alpha_im_i$ and an element $n\in \tau$ as $n=h(m_\tau)n_\sigma + \sum_{j=1}^d \beta_jn_j$, proves \eqref{eq:CompatibleChartsDual}. The second statement in the lemma follows in the same way. 
\end{proof}

\section{The c-gradient}\label{sec:cGradient}
As in \cite{HJMM}, and important role will be played by the $c$-gradient of a function $\Phi\in \mathcal P(\Delta)$. Let $\Phi\in \mathcal P(\Delta)$ and $n\in B$. The $c$-gradient of $\Psi$ at $n$ is the set of $m\in A$ such that 
$$ \Psi(n') \geq \Psi(n)+\langle m,n'-n \rangle $$
for all $n'\in B$. Similarly, if we define $\mathcal P(\Delta^\vee_h)$ to be the space of convex function $\Phi$ on $M_\mathbb R$ such that 
$$\sup_{m\in M_\mathbb R} |\Phi(m)-\sup_{n\in \Delta^\vee_h} \langle m,n \rangle|<\infty, $$
the $c$-gradient of $\Phi$ at a point $m\in A$ is the set of $n\in B$ such that
$$ \Phi(m') \geq \Phi(m)+\langle m'-m,n \rangle $$
for all $m'\in A$.

It might be useful to compare the $c$-gradient with the usual multivalued gradient $\partial$ of a convex function $f$ on $(\mathbb R^d)^*$. The gradient of $f$ at a point $y\in (\mathbb R^d)^*$ is the set of $x\in \mathbb R^d$ such that $f(y')\geq f(y)+\langle x,y'-y\rangle$ for all $y'\in (\mathbb R^d)^*$. The $c$-gradient shares many properties with this. In particular, the $c$-gradient is ($c$-) monotonous, i.e. if $m\in \partial^c\Psi(n)$ and $m'\in \partial^c\Psi(n')$ for $m,m'\in A$ and $n,n'\in B$, then
$$ \langle m'-m,n'-n \rangle \geq 0. $$ Moreover, as explained in Lemma~\ref{lem:GradientsInt} and Corollary~\ref{cor:GradientIso} below (which are the points of this section), the $c$-gradient of $\Psi$ can under favourable circumstances be identified with the usual gradient of the convex functions $\Psi|_\tau^\circ$, for $\tau$ a facet in $B$.

Let $\sigma$ be a facet of $\Delta$. For $n\in \Delta$, let $p_\sigma(n) = n+\rho n_\sigma$ where 
$$ \rho= \rho_{\sigma,n}:= \sup \rho': n+\rho' n_\sigma\in \Delta^\vee_h. $$
Similarly, given a facet $\tau$ of $\Delta^\vee_h$ and $m\in \Delta$, let 
$p_\tau(m) = m+\rho m_\tau$ where 
$$ \rho=\rho_{\tau,m}:=\sup \rho': m+\rho' m_\tau\in \Delta. $$
It follows from Lemma~\ref{lem:ProjectionDelta} and Lemma~\ref{lem:ProjectionDeltaDual} that $p_\sigma$ and $p_\tau$ define projections of $\Delta^\vee_h$ and $\Delta$ onto $\cup_{m\in \sigma\cap M} \tau_m$ and $\St(m_\tau)$, respectively. 
These projections have good properties with respect to the $c$-gradient. 
\begin{lemma}\label{lem:BetterGradient}
Let $\tau$ be a facet of $\Delta^\vee_h$ and $\Psi\in \mathcal P(\Delta)$. Assume $m\in A$, $n\in \tau$ and $m\in \partial^c\Psi(n).$ Then $p_\tau(m)\in \partial^c\Psi(n).$

Similarly, let $\sigma$ be a facet of $\Delta$ and $\Phi\in \mathcal P(\Delta^\vee_h)$. Assume $m\in \sigma$, $n\in B$ and $n\in \partial^c\Phi(m).$ Then $p_\sigma(n)\in \partial^c\Phi(m).$
\end{lemma}
\begin{proof}
Let $n'\in B$. We have
\begin{eqnarray*} 
\Psi(n') & \geq & \Psi(n)+\langle m,n'-n \rangle \\
& = & \Psi(n) + \langle p_\tau(m),n'-n \rangle - \langle m-p_\tau(m),n'-n \rangle \\
& = & \Psi(n) + \langle p_\tau(m),n'-n \rangle + \rho(\langle m_\tau,n'\rangle - \langle m_\tau,n\rangle) \\
& \geq & \Psi(n) + \langle p_\tau(m),n'-n \rangle
\end{eqnarray*}
since $\langle m_\tau,n \rangle = h(m_\tau) = \sup_{n'\in B} \langle m_\tau,n' \rangle.$
Hence $m'\in \partial^c\Phi(n)$. The second statement is proved in the same way. 
\end{proof}
\begin{corollary}\label{cor:GradientInStar}
    Assume $\tau$ is a facet of $\Delta^\vee_h$ and $n\in \tau$. Then $\partial^c\Psi(n)\cap \St(m_\tau)$ is non-empty.
\end{corollary}
\begin{proof}
By compactness of $\Delta$, $\partial^c\Psi(n)$ contains some $m\in \Delta$. By Lemma~\ref{lem:BetterGradient}, $p_\tau(m)\in \St(m_\tau)$ lies in $\partial^c\Psi(n)$.
\end{proof}

We now turn to the main lemma of this section.
\begin{lemma}[Comparing $\partial$ and $\partial^c$ on $B^\circ$]\label{lem:GradientsInt}
Let $\sigma$ be a facet of $\Delta$, $\tau$ be a facet of $\Delta^\vee_h$ and $(\alpha_\tau, \beta_\sigma)$ be $(\sigma,\tau)$-compatible.
Assume $\Psi\in \mathcal P(\Delta)$ and $n\in \tau^\circ$. Then
$$ \partial (\Psi\circ \beta^{-1}) (\beta(n)) \subset \alpha(\St(m_\tau)).$$
Moreover, the following two implications hold:
\begin{eqnarray}
    m\in \partial^c\Psi(n) & \implies & \alpha(m)\in \partial (\Psi\circ \beta^{-1})(\beta(n)) \label{eq:1stGradImplication}\\
    m\in \partial^c\Psi(n) & \impliedby & \alpha(m)\in \partial (\Psi\circ \beta^{-1})(\beta(n)) \text{ and } m\in \St(m_\tau). \label{eq:2ndGradImplication}
\end{eqnarray}

Similarly, assume $\Phi\in \mathcal P(\Delta^\vee_h)$ and $m\in \sigma^\circ$. Then
$$ \partial (\Phi\circ \alpha^{-1}) (\alpha(m)) \subset \beta(\cup_{m' \in \sigma\cap M} \tau_{m'}).$$
Moreover, the following two implications hold: 
\begin{eqnarray*}
    n\in \partial^c\Phi(m) & \implies & \beta(n)\in \partial (\Phi\circ \alpha^{-1})(\alpha(m)) \\
    n\in \partial^c\Phi(m) & \impliedby & \beta(n)\in \partial (\Phi\circ \alpha^{-1})(\alpha(m)) \text{ and } n\in \cup_{m' \in \sigma} \tau_{m'}.
\end{eqnarray*}
\end{lemma}
Before we prove this lemma we note the following straight forward corollary:
\begin{corollary}
    \label{cor:GradientIso}
    Let $\sigma,\tau,\alpha,\beta$ be as in Lemma~\ref{lem:GradientsInt}, $n\in \tau$ and $\partial^c\Psi(n) \subset \St(m_\tau)$. Then $\alpha$ defines a bijection of $\partial^c\Psi(n)$ onto $\partial (\Psi\circ \beta^{-1})(\beta(n))$.
\end{corollary}

\begin{proof}[Proof of Lemma~\ref{lem:GradientsInt}]
To prove the first claim in the lemma we write $\Psi=\Phi^*$ for $\Phi=\Psi^*\in L^\infty(\Delta)$. This means 
$$ \Psi(n) = \sup_{m\in \Delta} \langle m,n \rangle - \Phi(m). $$
It follows that for any $n\in \tau^\circ$ 
\begin{eqnarray} 
\Psi\circ\beta_\sigma^{-1}(\beta_\sigma(n)) & = & \Psi(n) \nonumber \\
& = & \sup_{m\in \Delta} \langle m,n \rangle - \Phi(m) \nonumber \\
& = & \sup_{m\in \Delta} \langle m,n-h(m_\tau)n_\sigma \rangle + \langle m, h(m_\tau)n_\sigma\rangle - \Phi(m) \nonumber \\
& = & \sup_{m\in \Delta} \langle \alpha_\tau(m),\beta_\sigma(n)\rangle + h(m_\tau)\langle m, n_\sigma\rangle - \Phi(m) \nonumber
\end{eqnarray}
where the fourth equality uses Lemma~\ref{lem:CompatibleCharts}. This means $\Psi\circ\beta^{-1}$ can be written as a pointwise supremum of affine functions whose linear part lies in $\alpha_\tau(\Delta) = \alpha_\tau(\St(m_\tau))$, proving the first claim in the lemma. 

We will now prove \eqref{eq:1stGradImplication}. Let $n\in \tau^\circ$ and $m\in \partial^c\Psi(n)$. This means 
$$    \Psi(n')\geq \Psi(n)+\langle m,n'-n\rangle$$ 
for all $n'\in B$. Consequently, 
\begin{eqnarray} 
\Psi\circ\beta_\sigma^{-1}(\beta_\sigma(n')) & = & \Psi(n') \nonumber \\
& \geq & \Psi(n) + \langle m,n'-n\rangle \nonumber \\
& = & \Psi(n) + \langle m,n'-h(m_\sigma)n_\tau\rangle - \langle m,n-h(m_\sigma)n_\tau\rangle \nonumber \\
& = & \Psi(n)+\langle \alpha_\tau(m),\beta_\sigma(n') \rangle - \langle \alpha_\tau(m),\beta_\sigma(n)\rangle \nonumber \\
& = & \Psi\circ\beta_\sigma^{-1}(\beta_\sigma(n))+\langle \alpha_\tau(m),\beta_\sigma(n')-\beta_\sigma(n)\rangle \label{eq:cGradImpliesGrad}
\end{eqnarray}
for all $n'\in \sigma^\circ$, hence $\alpha_\tau(m)\in \partial \Psi\circ\beta_\sigma^{-1}(\beta_\sigma(n))$. 

To prove \eqref{eq:2ndGradImplication}, let $n\in \tau^\circ$ and $m\in \St(m_\tau)$ and $\alpha_\tau(m)\in \partial(\Psi\circ \beta_\sigma^{-1})(\beta_\sigma(n))$. Then 
$$ \Psi\circ\beta_\sigma^{-1}(\beta_\sigma(n')) \geq \Psi\circ\beta_\sigma^{-1}(\beta_\sigma(n))+\langle \alpha_\tau(m),\beta_\sigma(n')-\beta_\sigma(n)\rangle $$
for all $ n'\in \tau^\circ.$ It follows that
\begin{eqnarray} 
\Psi(n') & = & \Psi\circ\beta_\sigma^{-1}(\beta_\sigma(n')) \nonumber \\
& \geq & \Psi\circ\beta_\sigma^{-1}(\beta_\sigma(n))+\langle \alpha_\tau(m),\beta_\sigma(n')-\beta_\sigma(n)\rangle \nonumber \\
& = & \Psi(n)+\langle m,n'-n_\sigma\rangle - \langle m,n-n_\sigma\rangle \nonumber \\
& = & \Psi(n)+\langle m,n'-n\rangle. \label{eq:GradImpliescGrad}
\end{eqnarray}
for all $n'\in \tau^\circ$. 

To conclude \eqref{eq:2ndGradImplication}, we need to prove that \eqref{eq:GradImpliescGrad} holds for all $n'\in B$. To do this, pick a facet $\sigma'$ in $\St(m_\tau)$ containing $m$. Pick $\beta_{\sigma'}$ so that $(\alpha_{\tau},\beta_{\sigma'})$ are $(\sigma',\tau)$-compatible. Using that \eqref{eq:GradImpliescGrad} holds for all $n'\in \tau^\circ$, a similar application of Lemma~\ref{lem:CompatibleCharts} as in \eqref{eq:cGradImpliesGrad} shows that 
\begin{equation} \label{eq:SubgradLocal} (\Psi-m_\tau)\circ\beta_{\sigma'}^{-1}(\beta_{\sigma'}(n')) \geq (\Psi-m_\tau)\circ\beta_{\sigma'}^{-1}(\beta_{\sigma'}(n))+\langle \alpha_\tau(m),\beta_{\sigma'}(n')-\beta_{\sigma'}(n)\rangle \end{equation}
for all $n'\in \tau^\circ$. Since $(\Psi-m_\tau)\circ\beta_{\sigma'}^{-1}$ is convex on $\beta_{\sigma'}(\cup_{m\in \sigma'} \tau_m)$ by Lemma~\ref{lem:ConvexPotential} and \eqref{eq:SubgradLocal} holds for all $n'\in \tau^\circ$ we get that \eqref{eq:SubgradLocal} holds for all $n'\in \cup_{m\in \sigma'} \tau_m$. By a similar application of Lemma~\ref{lem:CompatibleCharts} as in \eqref{eq:GradImpliescGrad}, this implies
$$ \Psi(n') \geq \Psi(n) + \langle m,n'-n \rangle $$
for all $n\in \cup_{m\in \sigma'} \tau_m$.

Now, assume $n'\in A$ and pick $m'\in \partial^c\Psi(\pi_{\sigma'}(n'))$, where $\pi_{\sigma'}(n')\in \cup_{m\in \sigma'} \tau_m$ is the image of $n'$ under the projection defined in the beginning of this section. We get
\begin{eqnarray*}
\Psi(n') & \geq & \Psi(\pi_{\sigma'}(n'))+\langle m',n'-\pi_{\sigma'}(n') \rangle \\
& \geq & \Psi(n)+\langle m,\pi_{\sigma'}(n')-n \rangle+\langle m',n'-\pi_{\sigma'}(n') \rangle \\
& = & \Psi(m) + \langle m,n'-n \rangle + \langle m'-m,n'-\pi_{\sigma'}(n') \rangle \\
& = & \Psi(m) + \langle m,n'-n \rangle - \rho\langle m'-m, n_{\sigma'}\rangle \\
& \geq & \Psi(m) + \langle m,n'-n \rangle,
\end{eqnarray*}
hence $m\in \partial^c\Psi(n)$. The corresponding statements for $\Phi$ are proved in the same way. 
\end{proof}

\section{Optimal Transport Plans}\label{sec:OptimalTransport}
We will now recall some notation and facts from the optimal transport literature. For details, we refer to any of the books \cite{Villani, AG, Figalli}. The standard setting in optimal transport is given by two (Polish) probability spaces and a lower semi-continuous \emph{cost function} defined on their product. Our setting consists of $A$ equipped with the positive, finite measure $\mu_M$ and $B$ equipped with a positive measure $\nu$ of the same total mass as $\mu_M$, together with the cost function $c:A\times B\rightarrow \mathbb R$ given by
$$ c(m,n) = -\langle m,n \rangle. $$
Although $\nu$ and $\mu_M$ are not probability measures, our setting is equivalent to the standard setting up to a harmless scaling. 

A \emph{transport plan} from $\nu$ to $\mu_M$ is a coupling of $\nu$ and $\mu_M$, i.e. a positive measure on the product $A\times B$ such that the pushforwards of $\gamma$ 
\begin{equation} \label{eq:MarginalsOTPlans} (q_A)\# \Gamma = \mu_M \text{ and } (q_B)_\# \gamma = \nu. \end{equation}
where $q_A:A\times B\rightarrow A$ and $q_B:A\times B\rightarrow B$ denote the projections onto the first and second factor of $A\times B$, respectively,

We will denote the set of transport plans from $\nu$ to $\mu_M$ by $\Pi(\nu,\mu_M)$. The optimal transport problem is to minimize
$$ I(\gamma) = \int_{A\times B} c\gamma $$
over $\Pi(\nu,\mu_M)$. It follows from lower semi-continuity of $I$ with respect tot he weak topology and Prokhorov's theorem that $I$ admits a minimizer. An \emph{optimal transport plan} is a transport plan that minimizes $I$. 

One of the most useful features of optimal transport is its striking interaction with its dual formulations. To explain this, we recall the following terminology from the optimal transport literature. For continuous functions $\Psi:\Delta^\vee_h\rightarrow \mathbb R$ and $\Phi:\Delta\rightarrow \mathbb R$, the $c$-transforms of $\Psi$ and $\Phi$ are the continuous functions $\Psi^c:A\rightarrow \mathbb R$ and $\Phi^c:B\rightarrow \mathbb R$ given by
\begin{equation} \label{eq:c-transform} \Psi^c(m) = \sup_{n\in B} \langle m,n \rangle-\Psi(n) \text{ and } \Phi^c(n) = \sup_{m\in A} \langle m,n \rangle-\Phi(m). \end{equation}
It follows from formal properties that any $\Psi$ and $\Phi$ as above satisfies $((\Psi^c)^c)^c = \Psi^c$ and $((\Phi^c)^c)^c = \Phi^c$. Those continuous functions $\Psi$ and $\Phi$ for which lie in the image of the $c$-transform, and thus satisfies $(\Psi^c)^c = \Psi$ and $(\Phi^c)^c = \Phi$, are called $c$-convex. 

A very useful feature, which arises formally, is that if $\Psi$ is $c$-convex and $(m,n)\in A\times B$, then $m\in \partial^c\Psi(n)$ if and only if $n\in \partial^c\Psi^c(m)$. The \emph{graph of the gradient} of $\partial^c\Psi$ is by definition the set of pairs $(m,n)\in A\times B$ satisfying these two properties. An equivalent definition of the graph of the gradient   can be made using the inequality 
\begin{equation} \label{eq:GraphOfGradient} \Psi(n)+\Psi^c(m) \geq \langle m,n \rangle \end{equation}
which holds for all $(m,n)\in A\times B$. The graph of the gradient is exactly those $(m,n)\in A\times B$ which attains equality in \eqref{eq:GraphOfGradient}. 

The dual problem of minimizing $I$ consists of maximizing the quantity 
$$ J(\Psi) = -\int_{A} \Psi^c\mu_M - \int_{B} \Psi \nu $$
over all $c$-convex functions $\Psi$. Concretely, using \eqref{eq:MarginalsOTPlans} and the fact that \eqref{eq:GraphOfGradient} holds
with equality if and only if $(m,n)$ lies on the graph of $\partial^c\Psi$ we get
\begin{equation} \label{eq:Duality} I(\gamma) = -\int \langle m,n \rangle\gamma \geq -\int(\Psi(n)+\Psi^c(m))\gamma = -\int \Psi\nu -\int \Psi^c \nu = J(\Psi) \end{equation}
and equality (which implies that both $\gamma$ and $\Psi$ are optimal) occurs if and only if $\gamma$ is supported on the graph of $\partial\Psi$. Moreover, strong duality holds and this condition is in fact necessary for optimality. 
We arrive at the following characterization of optimal transport plans which we will use multiple times in what follows: 

\noindent \textbf{Fact:} \emph{A transport plan is optimal if and only if it is supported on the graph of $\partial\Psi$ for a $c$-convex function $\Psi$} (see for example Theorem~5.10 in \cite{Villani}).

\begin{remark}
    Note that $\Psi^c$ is the unique function such that 
    $$ \Psi(n)+\Psi^c(m) \geq \langle m,n \rangle $$
    for all $(m,n)\in A\times B$ and for each $m\in A$ equality is achieved for some $n\in B$. It follows that to prove that a transport plan $\gamma$map is optimal, it suffices to produce a tcoupled of functions $(\Phi,\Psi)$ such that $\Psi(n)+\Phi(m)\geq \langle m,n\rangle$ throughout $A\times B$ with equality on the support of $\gamma$. 
\end{remark}

Note that the formulas in \eqref{eq:c-transform} can be used to define extensions of $\Psi^c$ and $\Phi^c$ to $M_\mathbb R$ and $N_\mathbb R$, respectively. We will identify these extensions with $\Psi^c$ and $\Phi^c$ and when convenient regard these as functions on $M_\mathbb R$ and $N_\mathbb R$. Note that since $(\Psi^c)^c = \Psi$ and $(\Phi^c)^c = \Phi$ for $c$-convex functions, we get a canonical extension of any $c$-convex function on $B$ or $A$ to $N_\mathbb R$ or $M_\mathbb R$, respectively. Conversely, any functions on $N_\mathbb R$ and $M_\mathbb R$ can be restricted to functions on $B$ and $A$, respectively. It turns out that these operations identify the sets of $c$-convex functions on $B$ and $A$ with $\mathcal P(\Delta)$ and $\mathcal P(\delta^\vee_h)$, as is proved in the following lemma. 
\begin{lemma}\label{lem:PDelta}
A lower semi-continuous function $\Psi:B\rightarrow \mathbb R$ is $c$-convex if and only if it is the restriction to $B$ of a function in $\mathcal P(\Delta)$.
\end{lemma}
\begin{proof}
Assume $\Psi$ is $c$-convex, i.e. $\Psi = \Phi^c$ for some $\Phi\in L^\infty(A)$. Let $C=\sup_A |\Phi|$. Then 
$$ \Psi(n) = \sup_{m\in A} \langle m,n\rangle -\Phi(m)  \leq  \sup_{m\in A} \langle m,n \rangle + C$$
and 
$$ \Psi(n) = \sup_{m\in A} \langle m,n\rangle -\Phi(m) \geq  \sup_{m\in A} \langle m,n \rangle - C$$
hence $\Psi\in \mathcal P(\Delta)$. 

Conversely, assume $\Psi\in \mathcal P(\Delta)$ and let $\Phi:M_\mathbb R\rightarrow \mathbb R$ be the Legendre transform of the convex, lower semi-continuous function $\Psi|_{\Delta_h^\vee}$, i.e.
\begin{equation} \label{eq:Phidef} \Phi(m) = \sup_{n\in \Delta^\vee_h} \langle m,n \rangle - \Psi(n). \end{equation}
By the involutive property of the Legendre transform, we get
\begin{equation} \label{eq:PhiTransform}\Psi(n) = \sup_{m\in M_\mathbb R} \langle m,n \rangle - \Phi(m). \end{equation}

We claim that for each $n'\in B$, there is $m'\in A$ such that $n'$ attains the supremum in \eqref{eq:Phidef}. To see this, let $n'\in B$, $\tau$ be a facet of $\Delta^\vee_h$ containing $n'$ and $m'\in \Delta$ be a subgradient of $\Psi$ at $n'$. It follows that $m'$ is a subgradient of $\Psi|_{\Delta^\vee_h}$ at $n'$. By standard properties of $\mathcal P(\Delta)$, $m'\in \Delta$. We get for any $n\in \Delta^\vee_h$
\begin{eqnarray*} \langle p_\tau(m'),n \rangle - \Psi(n) & = &  \langle m',n \rangle - \Psi(n) + \langle p_\tau(m')-m',n \rangle \\
& = & \langle m',n \rangle - \Psi(n) + \rho\langle m_\tau,n \rangle \\
& \leq & \langle m',n' \rangle - \Psi(n') + \rho\langle m_\tau,n' \rangle \\
& = & \langle \pi_\tau(m'),n' \rangle - \Psi(n')
\end{eqnarray*}
where the inequality follows from the fact that $n'\in B$ achieves the supremum in \eqref{eq:Phidef} and $\langle m_\tau,n \rangle\leq h(m_\tau) = \langle m_\tau,n' \rangle  $. 

Using the claim above, it follows from properties of the Legendre transform that for any $n\in B$, there is $m\in A$ attaining the supremum in \eqref{eq:PhiTransform}. Consequently,
\begin{eqnarray*} \Psi(n) & = & \sup_{m\in m_\mathbb R} \langle m,n \rangle - \Phi(m) \\
& = & \sup_{m\in A} \langle m,n \rangle - \Phi(m) 
\end{eqnarray*}
for any $n\in B$. It follows that $\Psi|_B = (\Phi|_A)^c$, hence $\Psi|_{B}$ is $c$-convex, finishing the proof of the lemma. 
\end{proof}

Since $\mathcal P(\Delta)$ is convex, one consequence of Lemma~\ref{lem:PDelta} is that the space of $c$-convex functions is convex. This is crucial in the following lemma.
\begin{lemma}\label{lem:OTUnique}
If $\Psi$ and $\Psi'$ are $c$-convex and satisfies 
$$ J(\Psi)=J(\Psi') = \sup J(\Psi) $$
then $\Psi-\Psi'$ is constant.
\end{lemma}
The functional $J$ is sometimes called the Kantorovich dual of $I$, after its original inventor. We will refer to the unique maximizer of $J$ as the \emph{Kantorovich potential} of $\nu$.
\begin{proof}[Proof of Lemma~\ref{lem:OTUnique}] 
Let $\Psi_t = t\Psi'+(1-t)\Psi$ and $\Phi_t=\Psi_t^c$. Note that, fixing $m\in A$, 
$$ \Phi_t(m) = \sup_{n\in \Delta^\vee_h} \langle m,n\rangle -\Psi_t(n) $$
is a supremum of affine functions (in $t$). It follows that $\Phi_t(m)$ is convex in $t$. Moreover, $\Phi_t(m)$ is affine in $t$ only if there exist $n_m\in \Delta^\vee_h$ such that 
$$ \Phi_t(m) =  \langle m,n_m\rangle -\Psi_t(n_m) $$
for all $t$, i.e. $\partial^c\Phi_0(m)\cap \partial^c\Phi_1(m)$ is non-empty. As $\Psi$ and $\Psi'$ are both maximizers of $J$, it follows that $J(\Psi_t)$ is affine in $t$, consequently $\partial^c\Phi_0(m)\cap \partial^c\Phi_1^c(m)$ is non-empty for almost all $m$ (with respect to $\nu$). 

Let $\sigma$ and $\tau$ be facets of $A$ and $B$ respectively and $\alpha$ and $\beta$ be $(\sigma,\tau)$-compatible charts. By Lemma~\ref{lem:GradientsInt}, it follows that $\partial(\Phi_0\circ \alpha^{-1})(\alpha(m))\cap \partial(\Phi_1\circ\alpha)(\alpha(m))$ is non-empty for almost all $m\in \sigma$. Since $\Phi_0\circ \alpha^{-1}$ and $\Phi_1\circ\alpha$ are convex, their subgradients are single valued almost everywhere. Consequently, $\partial(\Phi_0\circ \alpha^{-1})(\alpha(m)) = \partial(\Phi_1\circ\alpha)(\alpha(m))$ for almost all $m\in \sigma^\circ$. It follows (see for example \cite{HJMM}, Lemma~5.3) that $\Phi_0\circ\alpha-\Phi_1\circ\alpha$ is constant. Applying this to all facets $\sigma$ of $\Delta$ we get by continuity of $\Phi_0$ and $\Psi_1$ that $\Phi_0-\Psi_1$ is constant, and hence $\Psi-\Psi'$ is constant. 
\end{proof}

\section{Monge-Ampère Equation}\label{sec:MA}
Note that if $\Psi\in \mathcal P(\Delta)$ and $\sigma$ is a facet of $\Delta$, then $\Psi^c$ defines a convex function on $\sigma^\circ$. Given $\Psi\in \mathcal P(\Delta)$, we will let $R_\Psi\subset A^\circ$ be the set of $m\in \sigma^\circ\subset A^\circ$ such that $\partial \Psi^c|_{\sigma^\circ}$ is single valued. Note that since a convex function is differentiable almost everywhere we get $|\nu(R_\Psi\cap \sigma)|=|\alpha(R_\Psi\cap \sigma)| = |\alpha(\sigma)|=\nu(\sigma)$, hence $\nu(R_\Psi) = \nu(A)$. 
\begin{lemma}\label{lem:TMap}
There is a measurable function $T_\Psi:R_\Psi\rightarrow (\Delta_h^\vee)_d$ whose graph is the intersection of $R_\Psi\times (\Delta_h^\vee)_d$, \eqref{eq:WeaklyStableSet} and the graph of $\partial^c\Psi$.
\end{lemma}
\begin{proof}
Let $\sigma$ be a facet of $\Delta$ and $m\in R_\Psi\cap \sigma^\circ$. By Lemma~\ref{lem:GradientsInt}, if $n\in \cup_{m'\in \sigma\cap M} \tau_{m'}$, then $n\in \partial^c\Psi^c(m)$ if and only if 
$$ \beta_\sigma(n) \in \partial(\Psi|_\sigma \circ \alpha^{-1})(\alpha(n)).$$
By Corollary~\ref{cor:GradientInStar}, $\partial^c\Psi^c(m)\cap (\cup_{m'\in \sigma\cap M} \tau_{m'})$ is non-empty. 
As $\partial(\Psi|_\sigma \circ \alpha^{-1})$ is single valued on $\alpha(R_\Psi\cap \sigma^\circ)$ we get that the expression
$$ T_\Psi(m) = \beta_\sigma^{-1}(\partial(\Psi^c\circ \alpha^{-1})(\alpha(m)) $$
is well-defined. 

By construction, $T_\Psi(m)\in \cup_{m'\in \sigma\cap M} \tau_{m'}$ if $m\in \sigma\cap R_\Psi$, hence the graph of $T_\Psi$ is contained in \eqref{eq:WeaklyStableSet}. Moreover, by Lemma~\ref{lem:GradientsInt}, $m\in \partial^c \Psi(n)$ if $\partial (\Psi^c\circ \alpha^{-1})(\alpha(m)) = n$, hence the graph of $T_\Psi$ is contained in the graph of $\partial\Psi$. Trivially, the graph of $T_\Psi$ is contained in $R_\Psi\times B$. 

Conversely, assume $(m,n)$ lies in the intersection of \eqref{eq:WeaklyStableSet}, the graph of $\partial^c\Psi$ and $R_\Psi\times B$. As $m\in R_\Psi$ we get that $T_\Psi$ is defined at $m$ and that $m\in \sigma$ for some $\sigma\in A_d$. Since $(m,n)\in \eqref{eq:WeaklyStableSet}$ we get that $n\in \cup_{m'\in \sigma\cap M} \tau_{m'}$ and, since $n\in \partial \Psi^c(m)$, we can apply Lemma~\ref{lem:GradientsInt}, with $(\alpha_\tau,\beta_\sigma)$ choosen $(\sigma,\tau)$-compatible to get $\beta_\sigma(m)\in \partial (\Psi^c\circ \alpha^{-1})(\alpha_\tau(m))$. It follows that $n=T_\Psi(m)$, hence $(m,n)$ is in the graph of $T_\Psi$.
\end{proof}
\begin{corollary}\label{cor:TPushForward}
Assume $\Psi\in \mathcal P(\Delta)$ and $\gamma$ is a transport plan from $\nu$ to $\mu_M$ supported on the intersection of \eqref{eq:WeaklyStableSet} and the graph of $\partial^c\Psi^c$. Then
$$ (T_\Psi)_\# \mu_M = \nu. $$
\end{corollary}
\begin{proof}
Note that 
$$\gamma(R_\Psi\times B) = \nu(R_\Psi) = \mu_M(A) = \gamma(A\times B), $$
hence $\gamma$ is concentrated on $R_\Psi\times B$. Combined with Lemma~\ref{lem:GradientsInt} and the assumptions above, $\gamma$ is concentrated on the graph of $T_\Psi$. Let $E\subset \tau^\circ$ be a measurable set. For any $(m,n)$ in the graph if $T_\Psi$ we have $n\in E$ if and only if $m\in T^{-1}(E)$, hence
$$ \nu(E) = \gamma(A\times E) = \gamma(T^{-1}_\Psi(E)\times B) = \nu(T^{-1}_\Psi(E))$$
proving the corollary. 
\end{proof}

\begin{proof}[Proof of Theorem~\ref{thm:MAInt}]
Let $\Psi\in \mathcal P(\Delta)$ be the Kantorovich potential of $\nu$ and assume there is an optimal transport plan $\gamma$ from $\mu_M$ to $\nu$ concentrated on \eqref{eq:WeaklyStableSet}. We will now show that $\Psi$ satisfies \eqref{eq:MAInt}. Let $T_\Psi$ be the map furnished by Lemma~\ref{lem:TMap}. 

Assume $\sigma\in A_d$, $\tau\in B_d$ and $(\alpha_\tau,\beta_\sigma)$ be $(\sigma,\tau)$-compatible. Let $E\subset \tau^\circ$ be a measurable set. By Lemma~\ref{lem:GradientsInt}, we get 
$$\partial(\Psi\circ\beta^{-1})(\beta(E)) = \alpha(\partial^c\Psi(E)\cap \St(m_\tau))$$
and consequently
\begin{eqnarray}
|\partial(\Psi\circ\beta^{-1})(\beta(E))| & = & \mu_M(\partial^c\Psi(E)\cap \St(m_\tau)) \nonumber \\
& = & \mu_M(\partial^c\Psi(E)\cap \St(m_\tau) \cap R_\Psi) \label{eq:NegligibleLocus} \\
& = & \mu_M(T_\Psi^{-1}(E)) \label{eq:TInverseImage} \\
& = & \nu(E) \label{eq:TPushForward} 
\end{eqnarray}
where \eqref{eq:NegligibleLocus} and \eqref{eq:TInverseImage} follows from Lemma~\ref{lem:TMap} and \eqref{eq:TPushForward} follows from Corollary~\ref{cor:TPushForward}. 

Assume conversely that $\Psi\in \mathcal P(\Delta)$ satisfies \eqref{eq:MAInt}. We need to show that $\nu$ admits an optimal transport plan concentrated on \eqref{eq:WeaklyStableSet}. Let $T_\Psi$ be the map furnished by Lemma~\ref{lem:TMap}. Pick $\sigma\in A_d, \tau\in B_d$ such that $m_\tau\in \sigma$ and a measurable set $E\subset \tau^\circ$. Applying Lemma~\ref{lem:TMap} as above, we get
\begin{eqnarray}
\nu(E) & = & |\partial(\Psi\circ\beta^{-1})(\beta(E))| \nonumber \\
& = & \mu_M(\partial^c\Psi(E)\cap \St(m_\tau)) \nonumber \\
& = & \mu_M(\partial^c\Psi(E)\cap \St(m_\tau)\cap R_\Psi) \nonumber \\
& = & \mu_M(T_\Psi^{-1}(E)). \label{eq:TVolPreserving}
\end{eqnarray}
As $\nu$ doesn't charge $B_{d-1}$, it follows that 
$$\nu=(T_\Psi)_\#\mu_M|_{B^\circ}$$
and since $\nu$ and $(T_\Psi)_\#\mu_M$ have the same total mass, we may conclude that $\nu=(T_\Psi)_\#\mu_M$, hence $\gamma:=(id,T_\Psi)_\#\mu_M$ is a transport plan from $\nu$ to $\mu_M$. Moreover, $\gamma$ is supported on the graph of $T_\Psi$ which is contained in the graph of $\partial^c\Psi$ by construction of $T_\Psi$, hence $\gamma$ is optimal. Finally, since the graph of $T_\Psi$ is contained in \eqref{eq:WeaklyStableSet}, $\gamma$ is supported on \eqref{eq:WeaklyStableSet}. 

For uniqueness, let $\Psi, \Psi' \in \mathcal P(\Delta)$ be two solutions to \eqref{eq:MAInt}. By the construction above, there are optimal transport plans $\gamma$ and $\gamma'$ from $\nu$ to $\mu_M$ supported on the graphs of $\partial^c\Psi$ and $\partial^c\Psi'$ respectively, hence $\Psi$ and $\Psi'$ are both Kantorovich potentials of $\nu$, i.e. maximizers of $J$. By Lemma~\ref{lem:OTUnique}, $\Psi-\Psi'$ is constant.
\end{proof}

\begin{remark}
    Note that the map $T_\Psi$ exist and is single valued almost everywhere for any $\Psi\in \mathcal P(\Delta)$. As such, it plays the role of a more regular version of the usual $c$-gradient of $\Psi^c$. The main technical point in the proof of Theorem~\ref{thm:MAInt} is that $(T_\Psi)_\# \mu_M=\nu$ if and only if $(\Delta,h,\nu)$ is stable. 
\end{remark}

\section{Tropical Affine Structure and Monge-Ampère Equation Outside $B^\circ$}\label{sec:OutsideOpenFaces}
We will now explain how to extend the tropical affine structure and the Monge-Ampère equation \eqref{eq:MAInt} to a larger set than $B^\circ$. As explained in Section \ref{sec:AffineStructure}, there is for any facet $\sigma$ of $\Delta$ a natural coordinate function $\beta_\sigma$ defined on 
\begin{equation} \cup_{m\in \sigma} \tau_m \subset B\label{eq:Charts} \end{equation} 
(see Section~\ref{sec:AffineStructure}). These coordinate functions respect the tropical affine structure on $B^\circ$. However, if $\sigma'$ is another facet of $\Delta$, then $\beta_\sigma\circ \beta_{\sigma'}^{-1}$ is not affine on $\beta_{\sigma'}(B_{d-1})$ in general, where $B_{d-1} = B\setminus B^\circ$ is the union of the $d-1$-dimensional faces of $B$. Nevertheless, given a collection of pairwise disjoint sets $\{U_\sigma\}_{\sigma\in A_d}$ where $A_d$ is the set of facets of $\Delta$ and each $U_\sigma$ is a (possibly empty) open subset of \eqref{eq:Charts}, we get a tropical affine structure on $B^\circ\cup (\cup_{\sigma\in B_d} U_\sigma)$. We will call the complement of this set the \emph{singular set} and denote it 
$$ \Sigma = B_{d-1}\setminus (\cup_{\sigma\in A_d} U_\sigma) $$
Moreover, if $\Phi\in \mathcal P(\Delta)$, then by Lemma~\ref{lem:ConvexPotential} any $\sigma\in A_d$ and $m\in \sigma$ defines a function $(\Phi-m)|_{U_\sigma}$ on $U_\sigma$ which is convex with respect to the tropical affine structure on $B\setminus \Sigma$. This determines a Monge-Ampère measure $\MA((\Phi-m)|_{U_\sigma}))$ which agrees with $\MA(\Phi|_{\tau^\circ})$ on $\tau^\circ\cap U_\sigma$ for any facet $\tau$ of $\Delta^\vee_h$. Given the data $\nu$, $\{U_\sigma\}_{\sigma\in B_d}$ and $\Sigma$ this motivates the following Monge-Ampère equation on $B\setminus \Sigma$:
\begin{definition}\label{def:PDEproblem}
Let $\nu$ be a positive measure on $B$ and $\{U_\sigma\}_{\sigma\in A_d}$ a collection of charts with singular set $\Sigma$ as above. Then we will say that $\Phi\in \mathcal P(\Delta)$ solves the Monge-Ampère equation on $B\setminus \Sigma$ if 
$$ \MA(\Psi|_{B^\circ}) = \nu|_{B^\circ} $$
and for each facet $\sigma$ of $\Delta$ and $m\in \sigma$,
\begin{equation} \MA((\Psi-m)|_{U_\sigma}) = \nu|_{U_\sigma}. \label{eq:MA} \end{equation}
\end{definition}

A priori there is a lot of freedom in how to choose $\{U_\sigma\}_{\sigma\in A_d}$.
When $\Delta$ is the standard unit simplex, $h=h_0$ and $\nu$ is invariant under permutation of the vertices of $n_0,\ldots n_{d+1}$ of $\Delta^\vee_h=\delta^\vee$ then, arguing by symmetry, a natural choice for $\{U_\sigma\}_{\sigma\in A_d}$ and $\Sigma$ is
$$ U_{\sigma_j} := \left\{ \sum_{i=0}^{d+1} \beta_i n_i\in B, \beta_j>\beta_i \text{ for all } i\not=j\right\} $$
where, for each $i$, $\sigma_i$ is the facet of $\Delta$ dual to the vertex $n_i$ (see \cite{LiFermat,HJMM}). When $d=2$, this choice makes $\Sigma$ the midpoints $\{(n_i+n_j)/2\}_{i\not=j}$ of the edges of $\Delta^\vee_h$ and for $d\geq 2$, $\Sigma$ is the codimension 2 set 
$$ \Sigma = \left\{ \sum_{i=0}^{d+1} \beta_i
n_i\in B_{d-1}, \beta_j=\beta_k=\max \beta_i, \text{ for some distinct } i,j\right\}. $$
When $d=3$, so called $Y$-shaped singularities appear. In all these examples the codimension of $\Sigma$ is 2. 

One purpose of the present paper is to argue that in the absence of symmetry, $\{U_\sigma\}_{\sigma\in B_d}$ and $\Sigma$ needs to be chosen to suit $\nu$, and thus that $\Sigma$ plays the role of a free boundary in the PDE problem of Definition~\ref{def:PDEproblem}. In this spirit, given $\nu$, we propose a choice of $\{U_\sigma\}_{\sigma\in B_0}$ and $\Sigma$ based on the Kantorovich potential of $\nu$, i.e. the unique minimizer of $I$ (see Section~\ref{sec:OptimalTransport}). 

\begin{definition}\label{def:Usigma}
Let $\nu$ be a positive measure on $B$ of mass $\mu_M(A)$. We define $\{U_\sigma\}_{\sigma\in A_d}$ and $\Sigma$ as 
$$ U_\sigma = U_{\sigma,\Psi}:= B\setminus \partial^c\Psi^c(A\setminus \sigma_n^\circ) $$
and $\Sigma=\Sigma_\Psi := B_{d-1}\setminus (\cup U_\sigma)$,
where $\Psi$ is the Kantorovich potential of $\nu$.
\end{definition}
It follows from continuity of $\Psi, \Psi^c$ and compactness of $\Delta$ that $\partial^c\Psi^c(A\setminus \sigma_n^\circ)$ is closed, hence $U_n$ is open. Moreover, by monotonicity properties of the $c$-gradient, $U_n\subset \St(n)^\circ$ (see Lemma~\ref{lem:OpenSetsForCharts}). 
The main motivation for our choice of $\{U_\sigma\}_{\sigma\in B_d}$ and $\Sigma$ is given by Lemma~\ref{lem:SingularOutsideCharts} and Theorem~\ref{thm:MA} below. The first of these states that the solution to \eqref{eq:MAInt} can be extended to a solution on $B\setminus \Sigma$ in the sense of Definition~\ref{def:PDEproblem} and the second addresses the second bullet point regarding $\Sigma_\Psi$ in the introduction. 

\begin{lemma}\label{lem:OpenSetsForCharts}
Let $\Phi\in \mathcal P(\Delta)$, $\sigma$ be a facet of $\Delta$ and $U_\sigma$ the corresponding chart in Definition~\ref{def:Usigma}. Then $U_\sigma\subset \cup_{m\in \sigma\cap M} \tau_m$.
\end{lemma}
\begin{proof}
Assume $n\notin \cup_{m\in \sigma\cap M} \tau_m$. This means $n\in \tau$ for some $\tau$ such that $m_\tau\notin \sigma$. It follows that $\sigma$ is not a subset of $\St(m_\tau)$. 
On the other hand, by Corollary~\ref{cor:GradientInStar}, $\partial^c\Psi(n)$ contains some $m'\in \St(m)$. It follows that $m'\in \sigma'$ for some $\sigma'\not=\sigma$, hence $n\notin U_\sigma$. 
\end{proof}

\begin{theorem}\label{thm:MA}
Let $\nu$ be a positive measure on $B$ and $\Psi$ its Kantorovich potential. 
If there exist an optimal transport plan from $\mu_M$ to $\nu$ which is supported on 
\begin{equation} \label{eq:StabilityThmMA} \cup_{m\in A\cap M} \left(\St(m) \times \tau_m\right)\end{equation}
then $\Psi$ solves the Monge-Ampère equation in Definition~\ref{def:PDEproblem} on $B\setminus \Sigma_\Psi$. 
Moreover, if $\Psi\in \mathcal P(\Delta)$ solve the Monge-Ampère equation in Definition~\ref{def:PDEproblem} and $\nu$ doesn't charge $\Sigma_\Psi$ then there is an optimal transport plan from $\mu_M$ to $\nu$ supported on \eqref{eq:StabilityThmMA}, and if $\Psi'$ is another solution to Definition~\ref{def:PDEproblem} such that $\nu$ doesn't charge $\Sigma_{\Psi'}$, then $\Sigma_\Psi = \Sigma_{\Psi'}$ and $\Psi-\Psi'$ is constant. 
\end{theorem}

To prove Theorem~\ref{thm:MA} we will first state and prove a lemma which will play the same role for the charts $U_\sigma$ as Lemma~\ref{lem:GradientsInt} plays for the open faces $\tau^\circ$. 
\begin{lemma}[Comparing $\partial$ and $\partial^c$ on $U_\sigma$]\label{lem:GradientsEdge}
Assume $\Phi\in P(\Delta)$, $\sigma$ is a facet of $\Delta$ and $\tau$ is a facet of $\Delta^\vee_h$. Pick $U_\sigma$ as in Definition~\ref{def:Usigma} and $(\alpha_\tau,\beta_\sigma)$ which are $(\sigma,\tau)$-compatible. Assume $m_0\in \sigma$ and $n\in U_\sigma$. Then $\alpha_\tau$ defines a bijection of $\partial^c\Psi(n)$ onto $\partial((\Psi-m_\tau)\circ \alpha_\tau^{-1})$.
\end{lemma}
\begin{proof}
First of all, by definition $U_\sigma=B\setminus \partial^c\Psi^c(A\setminus \sigma^\circ)$, hence $n\in U_\sigma$ if and only if there is no $m\in A\setminus \sigma^\circ$ such that $\partial^c\Psi^c(m)\ni n$, or equivalently $m\in \partial^c\Psi(n)$. It follows that $\partial^c\Psi(U_\sigma)\subset \sigma^\circ$. 

As in the first claim of Lemma~\ref{lem:GradientsInt}, it follows that $\partial((\Psi-m_\tau)\circ \alpha_\tau^{-1})(\alpha(n))\subset \alpha(\sigma^\circ)$ for any $n\in U_\sigma$. To see this, write $\Psi=\Phi^*$ for $\Phi\in L^\infty(\Delta)$ and note that since $\partial^c\Psi(U_\sigma)\subset \sigma^\circ$
\begin{eqnarray*}
    (\Psi-m_\tau)\circ \alpha_\tau^{-1})(n) & = & \Psi(n)-\langle m_0,n \rangle \\
    & = & \sup_{m\in \Delta} \langle m-m_\tau,n \rangle - \Phi(m) \\
    & = & \sup_{m\in \Delta} \langle \alpha_\tau(m),\beta_\sigma(n) \rangle - \Phi(m),
\end{eqnarray*}
hence that $(\Psi-m_\tau)\circ \alpha_\tau^{-1})$ can be written as a pointwise supremum of affine functions whose linear part lies in $\alpha_\tau(\sigma^\circ).$

It now suffices to prove that if $m\in \sigma^\circ$ and $n\in U_\sigma$, then $\partial((\Psi-m_\tau)\circ \alpha_\tau^{-1})(\alpha(n))$ if and only if $m\in \partial^c\Psi(n)$. This follows similarly as in Lemma~\ref{lem:GradientsInt}. More precisely, let $m\in \partial^c\Psi(n)$. We have
\begin{eqnarray*}
    (\Psi-m_\tau)\circ \beta_\sigma^{-1}(\beta_\sigma(n')) & = & \Psi(n')-\langle m_\tau,n' \rangle \\
    & \geq & \Psi(n) - \langle m_\tau,n \rangle + \langle m,n-n' \rangle - \langle m_\tau,n'-n \rangle  \\
    & = & \Psi(n) - \langle m_\tau,n \rangle + \langle m-m_\tau,n' \rangle - \langle m-m_\tau,n \rangle \\
    & = & (\Psi-m_\tau)\circ \beta_\sigma^{-1}(\beta_\sigma(n)) + \langle \alpha_\tau(m),\beta_\sigma(n')-\beta_\sigma(n) \rangle,
\end{eqnarray*}
hence $\alpha(m)\in \partial((\Psi-m_\tau)\circ \beta_\sigma^{-1})(\beta_\sigma(n))$. The converse follows in the same way. 
\end{proof}

\begin{proof}[Proof of Theorem~\ref{thm:MA}]
Assume $\nu$ admits an optimal transport plan $\gamma$ to $\mu_M$ supported on \eqref{eq:WeaklyStableSet}. A similar computation as in the proof of Theorem~\ref{thm:MAInt}, equations \eqref{eq:NegligibleLocus}, \eqref{eq:TInverseImage} and \eqref{eq:TPushForward}, gives that the the Kantorovich potential $\Psi$ satisfies \eqref{eq:MAInt}. 
It remains to check \eqref{eq:MA}. 
Let $\sigma\in A_d,\tau\in B_d$ such that $m_\tau\in \sigma$ and $E\subset U_\sigma$ be measurable. We get by Lemma~\ref{lem:GradientsEdge} 
$$ \partial((\Psi-m_0)\circ \beta)(\beta(E)) = \alpha(\partial^c(E)) $$
and consequently
\begin{eqnarray}
    |\partial((\Psi-m_0)\circ \beta)(\beta(E))| & = & \mu_M(\partial^c\Psi(E)) \nonumber \\
    & = & \mu_M(\partial^c\Psi(E)\cap R_\Psi) \label{eq:NeglibleLocus2}\\
    & = & \mu_M(T^{-1}(E)) \label{eq:TInverseImage2}\\
    & = & \nu(E). \label{eq:TPushForward2} 
\end{eqnarray}

Assume now that $\Psi\in \mathcal P(\Delta)$ satisfies the Monge-Ampère equation in Defintion~\ref{def:PDEproblem} and $\nu$ does not charge $\Sigma_\Psi$. As in the proof of Theorem~\ref{thm:MAInt}, let $T_\Psi$ be the map furnished by Lemma~\ref{lem:TMap}. 
For a measurable set $E\subset U_n$, we have 
\begin{eqnarray*}
\nu(E) & = & |\partial((\Psi-m)\circ \beta^{-1})(\beta(E))| \\
& = & \mu_M(\partial^c\Psi(E)) \\
& = & \mu_M(\partial^c\Psi(E)\cap R_\Psi) \\
& = & \mu_M(T^{-1}_\Psi(E)).
\end{eqnarray*}
As $\nu$ doesn't charge $\Sigma_\Psi$, it follows that 
$$ \nu = (T_\Psi)_\#\mu_M|_{B^\circ} $$
and since $\nu$ and $(T_\Psi)_\#\mu_M|_{B^\circ}$ have the same total mass, we conclude that $\nu = (T_\Psi)_\#\mu_M$ and $\gamma:=(id,T_\Psi)_\#\mu_M$ provides an optimal transport plan supported on \eqref{eq:WeaklyStableSet} as in the proof of Theorem~\ref{thm:MAInt}. 

For uniqueness, let $\Psi,\Psi'\in \mathcal P(\Delta)$ be two solutions to the Monge-Ampère equation in Definition~\ref{def:PDEproblem} such that $\nu$ doesn't charge $\Sigma_\Psi$ or $\Sigma_{\Psi'}$. By the construction above, there are optimal transport plans from $\nu$ to $\mu_M$ supported on the graphs of $\partial^c\Psi$ and $\partial^c\Psi'$, respectively, hence $\Psi$ and $\Psi'$ are both Kantorovich potentials of $\nu$, i.e. maximizers of $J$. By Lemma~\ref{lem:OTUnique}, $\Psi-\Psi'$ is constant. 
\end{proof}

We will now state and prove a lemma which essentially answers the second bullet point regarding $\Sigma_\Psi$ in the introduction (uniqueness, assuming minimality). The lemma says that if $\Sigma_\Psi$ is minimal and $\Sigma'$ is the singular set of another tropical affine structure defined as in the discussion preceding Definition~\ref{def:PDEproblem}, then, although a weak Monge-Ampère equation might be satisfied on $B\setminus \Sigma$, the solution will not be differentiable. 
\begin{lemma}\label{lem:SingularOutsideCharts}
Assume $\Sigma_\Psi$ is minimal, i.e. there is no closed set $\Sigma'\subset \Sigma_\Psi$ such that the tropical affine structure on $B\setminus \Sigma$ extends to $B\setminus \Sigma'$ and the Monge-Ampère equation holds on $B\setminus \Sigma'$. Let $\{U'_\sigma\}_{\sigma\in A_d}$ and $\Sigma'$ be a set of charts and a singular set such that $\Sigma_\Psi$ is not a subset of $\Sigma'$. Then there is $\sigma$ and $n\in U_{\sigma}'$ such that $(\Psi-m_0)|_{U_\sigma'}\circ \beta_\sigma^{-1}$ is not differentiable at $n$ for $m_0\in \sigma$.
\end{lemma}
\begin{proof}
It follows from the assumptions of the lemma that there is a facet $\sigma$ of $\Delta$ and $n\in U'_\sigma\cap \Sigma_\Psi\subset B_{d-1}$. Now, either $n$ has a neighbourhood $U_n$ in $B_{d-1}$ such that $\partial \Psi(U_n)\subset \sigma$ or there is $n'\in U'_\sigma$ and $m\in \partial^c \Psi(n)$ such that $m\notin \sigma$. If the former is true, then $U_n$ does not intersect $U_{\sigma'}$ for any facet $\sigma'$ different from $\sigma$. 
Moreover, using Lemma~\ref{lem:GradientsEdge} in the same way as in the proof of Theorem~\ref{thm:MA}, it is possible to prove that $(\Psi-m)\circ \alpha^{-1}_\sigma|_{U_n}$ satisfies the Monge-Ampère equation on $U_n$, hence $U_n$ can be added to $U_\sigma$, contradicting the minimality of $\Sigma_\Psi$. 

In the latter case, we have $m\in \partial^c\Psi(n')$.
We will use the notation from the proof of Lemma~\ref{lem:ConvexPotential}. Let $m_0\in \sigma$. Note that $(\Psi-m_0)\circ \beta^{-1}\geq (m-m_0)\circ \beta^{-1}$ with equality at $\beta(n')$, hence it suffices to show that $(m-m_0)\circ \beta^{-1}$ is convex and not differentiable at $\beta(n')$, or equivalently that $(m-m_0)\circ l$ is not differentiable at $n'\in H$. But this follows immediately, since 
\begin{eqnarray*} 
(m-m_0)\circ l(n) & = &  \langle m-m_0,n-h(n')n_\sigma \rangle \\ 
& = & \langle m-m_0, n' \rangle - h(n')\langle m-m_0, n_\sigma \rangle, 
\end{eqnarray*}
$h$ is a piecewise affine convex function whose non-differentiable locus is exactly the codimension 1 skeleton of $\cup_{m\in \sigma\cap M} \tau_m$, hence contains $n'$, and $\langle m-m_0, n_\sigma \rangle<0$ since $m\notin \sigma$. 
\end{proof}

\begin{remark}
    Controlling the size of $\Sigma_\Psi$ is related to regularity theory for $\Psi$. Interesting developments regarding regularity of solutions to the Monge-Ampère equation in relation to mirror symmetry are given in \cite{Moo21,MR22}.
\end{remark}

We will now state and prove a theorem which gives a sufficient condition for smoothness of the solution on $B\setminus \Sigma_\Psi$. 

\begin{theorem}\label{thm:SmoothSolutions}
Consider stable data $(\Delta,h_0,\nu_N)$ (or more generally, $(\Delta,h_0,\nu)$ such that for all facets $\sigma$ of $\Delta$, $(\beta_{\sigma})_\# \nu|_{\mathrm{St}(\sigma)}$ is absolutely continuous, and the density is smooth, uniformly non-negative and uniformly bounded). 

Assume for any facet $\tau$ of $\Delta^\vee$ that
\begin{equation*}
    \alpha_\tau\left(\overline{\partial^c \Psi(\tau^\circ)}\right)
\end{equation*}
is convex. 
Then $\Psi$ is smooth on $B^\circ$.

If in addition, for any facet $\sigma$ of $\Delta$, the set 
$$ 
F_\sigma := \overline{\partial^c\Psi^c(\sigma^\circ)}
%\beta_\sigma\left(\overline{\partial^c\Psi^c(\sigma^\circ)}\right) 
$$
satisfies $F_\sigma\subset \St(n_\sigma)$, $\partial^c\Psi(F^\circ_\sigma) \subset \sigma$ and $\beta_\sigma(F_\sigma)$ is convex,  
then $U_\sigma = F_\sigma^\circ$
%$$ U_\sigma = %\left(\overline{\partial^c\Psi^c(\sigma^\circ)}\right)^\circ $$
and $\Psi-m$ is smooth on $U_\sigma$ for each
facet $\sigma$ of $\Delta$ and $m\in \sigma$, i.e. $\Psi$ defines a smooth Hessian metric on $B\setminus \Sigma_\Psi$. 
\end{theorem}

\begin{remark}
We will show in Lemma~\ref{lem:Symmetric} and Lemma~\ref{lem:SymmetricCube} how symmetries of the data $(\Delta,h,\nu)$ can be used to verify the conditions in Theorem~\ref{thm:SmoothSolutions}.
\end{remark}

\begin{proof}[Proof of Theorem \ref{thm:SmoothSolutions}]%\textit{(of Theorem \ref{thm:SmoothSolutions})}

For the first part, note that stability of the data implies that $u:=\Psi\circ \beta_\sigma^{-1}$ restricted to $\beta_\sigma(\tau)$ solves the Monge-Ampère equation in the sense of Remark~\ref{rem:WeakSolution}, i.e. in the Aleksandrov sense. 
Moreover, by Lemma~\ref{lem:GradientsInt} 
$$ \overline{\partial u(\beta_\sigma(\tau^\circ))} = \alpha_\tau\left(\overline{\partial^c \Psi(\tau^\circ)\cap \St(m_\tau)}\right) = \alpha_\tau(\overline{\partial \Psi(\tau^\circ)}), $$
where the second equality is given by Lemma~\ref{lem:ProjectionDelta}. This set is convex by assumption. It thus follows by \cite{Caf97} that $u$ is smooth on $\beta_\sigma(\tau^\circ)$ and $\Psi$ is smooth on $\tau^\circ$.

For the second part, note that since $F_\sigma\subset \St(n_\sigma)$, applying Lemma~\ref{lem:GradientsInt} in the same way as in the proof of Theorem~\ref{thm:MAInt}, we get that $\Phi=\Psi^c$ satisfies a Monge-Ampère equation with smooth, uniformly non-negative and uniformly bounded density on $\sigma^\circ$. Moreover, since $\beta_\sigma(F_\sigma)$ is convex, $v:=\Phi\circ \alpha_\tau^{-1}$ is smooth and strictly convex. Since 
in addition $\partial^c\Psi(F^\circ_\sigma) \subset \sigma$, it follows from Lemma~\ref{lem:GradientsInt} and Lemma~\ref{lem:GradientsEdge} that 
$\partial^c\Psi$ is single valued and injective on $F_\sigma^\circ$. Moreover, since $\partial^c\Psi(F_\sigma^\circ)\subset \sigma$, applying Lemma~\ref{lem:GradientsEdge} in the same way as in the proof of Theorem~\ref{thm:MA}, we get that 
$$ u_\sigma = (\Psi-m)\circ \beta^{-1}_\sigma $$
satisfies the Monge-Ampère equation in the sense of Aleksandrov on $\beta_\sigma(F_\sigma^\circ)$. As $\partial^c\Psi$ is single valued and injective on $F_\sigma^\circ$ we can apply Lemma~\ref{lem:GradientsEdge} to see that $u_\sigma$ is strictly convex and hence smooth on $F_\sigma^\circ$.

To see that $U_\sigma = F_\sigma^\circ$, assume $n\in F_\sigma^\circ$ and note that by smoothness of $u_\sigma$, $\partial^c \Psi(n)$ is single valued and contained in $\sigma^\circ$, hence $n\in U_\sigma$. Conversely, assume $n\in U_\sigma$. It follows that $\partial^c\Psi(n)\subset \sigma^\circ$. Let $m\in \partial^c\Psi(n)\cap \sigma^\circ$. Since the gradient of the Legendre transform of $u_\sigma$ maps $\alpha_\tau(\sigma^\circ)$ homeomorphically onto $\beta_\sigma(F_\sigma^\circ)$ we get that $\beta_\sigma(n)\in \beta_\sigma(F_\sigma^\circ)$, hence $n\in F_\sigma^\circ$.
\end{proof}

\section{Examples}\label{sec:Examples}
As in previous sections we will let $e_0,\ldots,e_d$ be a set of generators of $M$ and $f_0,\ldots,f_d$ be a set of generators of $N$. We will use $h_0$ to denote the trivial height function defined by $h_0(0)=0$ and $h_0(m) = 1$ for $m\not= 0$. For convenience when making explicit computations we will also normalize the total mass of $\mu_M$, and hence $\nu_N$, to be 1.

\subsection{Existence and smoothness in the presence of discrete symmetry}\label{sec:Examples Existence}
Stability, and hence existence of solutions, can often be verified if $(\Delta,h,\nu)$ has strong symmetry properties. In the following two lemmas we show how this can be applied to the standard unit simplex and the unit cube. Moreover, in these examples the symmetries give enough control on the gradient to show that the solutions are smooth when $\nu=\nu_N$, using Theorem~\ref{thm:SmoothSolutions}. Symmetry was similarly exploited in \cite{LiFermat,HJMM} to prove existence of (weak) solutions on the unit simplex.

\begin{lemma}\label{lem:Symmetric}
Let $\Delta$ be the standard unit simplex 
$$ \Delta = \conv\left\{ (d+1)e_0-\sum_{i \neq 0} e_i, \ldots, (d+1)e_d-\sum_{i\neq d} e_i, -\sum_{i=0}^d e_i\right\} $$
and $h = h_0$. Assume that $\nu$ is invariant under permutations of the vertices of $B$ and don't charge $B_{d-1}$. Then the Monge-Ampère equation \eqref{eq:MAInt} admits a solution. 

If in addition $\nu=\nu_N$ (or, more generally, for any $\sigma\in \mathcal{V}(\Delta^\vee)$, $(\beta_\sigma)_\# \nu|_{\mathrm{St}(\sigma)}$ is absolutely continuous, and the density is smooth, uniformly non-negative and uniformly bounded) then the solution is smooth outside of the set

$$\Sigma_s := \{x\in B: \beta_k=\beta_l=0, \beta_i=\beta_j=\max_p \beta_p,k\neq l \neq i \neq j\}$$

where $\beta_i$ is defined by uniquely writing $x$ as a convex combination of the vertices of $\Delta^\vee$, i.e. $x=\sum_{i=0}^{d+1} \beta_i n_i$ where the $n_i$'s are the vertices of $\Delta^\vee$ and $\sum_{i=0}^{d}\beta_i =1,\beta_i\geq0 \forall i$ (when $d=1$, $\Sigma_s$ is empty and when $d=2$, $\Sigma_s$ is the midpoint of each of the edges). Additionally $\Sigma_s=\Sigma_\Psi$.

\end{lemma}
\begin{proof}
Let $m_0,\ldots,m_{d+1}$ be the vertices of $\Delta$ ordered as in the statement of the lemma and let $n_0,\ldots,n_{d+1}$ be the vertices of $\Delta^\vee$, ordered as in the description below:
$$ \Delta^\vee = \conv\left\{-f_0,\ldots,-f_d ,\sum_{i=0}^d f_i\right\}.$$
We have $\langle m_i,n_j \rangle = 1$ if $i\not=j$ and $\langle m_i,n_i \rangle = -(d+1)$ for each $i$.

Now, let $\Psi$ be the Kantorovich potential of $\nu$. By uniqueness of maximizer of $J$, $\Psi$ is symmetric. It follows that the graph of $\partial^c\Psi$ is symmetric. We claim that the graph of $\partial^c\Psi$ lies in the set

\begin{align}\label{eq: regularly stable set1}
    \cup_{i=0}^d \tau_{m_i} \times \mathrm{St}(m_i).
\end{align}

To see this, assume $(m,n)$ lies in the graph of $\partial^c\Psi$ but not in \eqref{eq: regularly stable set1}. Without loss of generality,  assume that
$n\in \tau_{m_0} = \conv\{ n_1,\ldots,n_d\} $
and $m\in \sigma_{n_0}^\circ = \conv\{m_1,\ldots,m_d\}^\circ$. It follows that 
$$ m = \sum_{i=1}^d \alpha_i m_i \;\;\;\; n = \sum_{i=1}^d \beta_i n_i $$
for some $\alpha_1,\ldots,\alpha_d > 0$ and $\beta_1,\ldots,\beta_d\geq 0$ such that $\sum \alpha_i = \sum \beta_i = 1$. We have $\beta_j>0$ for some $j$. Let $F_{0j}$ and $G_{0j}$ be the maps on $A$ and $B$ interchanging $m_0$ with $m_j$ and $n_0$ and $n_j$, respectively. By symmetry, we get that the pair $(F_{0j}(m),G_{0j}(n))$ is also in the graph of $\partial^c\Psi$. However,
\begin{eqnarray*}
\langle F_{0j}(m)-m,G_{0j}(n)-n \rangle & = & \langle \alpha_j(m_0-m_j),\beta_j(n_0-n_j) \rangle \\
& = & -\alpha_j\beta_j(2d+4)  \\
& < & 0 
\end{eqnarray*}
contradicting the monotonicity of $\partial^c\Psi$. For existence of a solution to the Monge-Ampére equation \ref{eq:MAInt} in the interior of the facets, simply consult Theorem \ref{thm:MAInt}. 

For the issue of regularity, we aim to consult Theorem \ref{thm:SmoothSolutions}. Namely we claim that, for each facet $\tau_{m_i}$ 

\begin{equation}\label{eq: regularly stable set}
    \partial^c \tau_{m_i}^\circ \subset \{x\in \mathrm{St}(m_i): \alpha_i\geq \alpha_j, \forall j\}:=\mathrm{SmSt(m_i)}.
\end{equation}

Here, $\alpha_i$ is defined by writing uniquely $x=\sum_{i=0}^{d+1} \alpha_i m_i$ with $\sum_{i=0}^{d}\alpha_i =1,\alpha_i\geq0 \forall i$. To see the claim, assume the contrary for the point $(n,m)$. Without loss of generality, there are two cases. In the first case, $n\in \tau_{m_0} = \conv\{ n_1,\ldots,n_d\} $
and $m\in \sigma_{n_0}^\circ = \conv\{m_1,\ldots,m_d\}^\circ$, but that this contradicts monotonicity of the graph of $\partial^c\Psi$ follows from our earlier arguments. 

In the second case, we assume that $n\in \tau_{m_0}=\conv\{n_1,...,n_d\}^\circ$ and $m\in \conv\{m_0,m_1,...,m_{d-1} \}$ and additionally that after writing
\begin{equation*}
    m = \sum_{i=0}^{d-1} \alpha_i m_i 
\end{equation*}
we have $\alpha_0<\alpha_j$ for some $j\in\{1,\ldots d-1\}$, $\sum_{i=0}^{d-1} \alpha_i = 1$ and $\alpha_i\geq 0 \ \forall i=0,1,\ldots d$. We also write 
\begin{equation*}
    n = \sum_{i=1}^d \beta_i n_i
\end{equation*}
with $\sum_{i=1}^d \beta_i=1$ and $\beta_i> 0\ \forall i=1,\ldots, d$.  Let $F_{0j}$ and $G_{0j}$ be defined as before. By symmetry, we get that the pair $(F_{0j}(m),G_{0j}(n))$ is also in the graph of $\partial^c\Psi$ and in addition $F_{0j}(m)\in \tau_{m_j}^\circ$. However,

\begin{align*}
    \langle F_{0j}(m)-m,G_{0j}(n)-n\rangle &= \langle \alpha_0(m_j-m_0)+\alpha_j(m_0-m_j),\beta_j(n_0-n_j)\rangle \\
    &= -(\alpha_j-\alpha_0)\beta_j(2d+4)<0
\end{align*}
contradicting the monotonicity of $\partial^c \Psi$. Next note that the open sets $\tau_{m_i}^\circ$ and $\mathrm{SmSt}(m_i)^\circ$ have the same mass by symmetry considerations and using that neither $\nu$ nor $\mu$ charges sets of codimension 1. Thus 

\begin{equation*}
    \overline{\alpha_{m_i}(\partial^c\Psi(\tau_{m_i}^\circ))}=\alpha_{m_i}(\mathrm{SmSt}(m_i))
\end{equation*}
which is convex, the proof of which we postpone to Lemma \ref{lem:convexity in coords}, so we can use the first part of Theorem \ref{thm:SmoothSolutions}. 

To obtain regularity on a larger set, we use an additional symmetry available in the case of the standard simplex. While $\Delta$ is not unimodularly self-dual it is  affinely self-dual. Even more, the affine map mapping $\Delta$ bijectively to $\Delta^\vee$ multiplies the lattice-induced volume element of each facet by one and the same constant. Exchanging the roles of $\Delta$ and $\Delta^\vee$ via this symmetry we have thus shown that for any facet $\sigma$ of $\Delta$, the set 
\begin{equation*}
    \beta_\sigma(\overline{\partial^c\Psi^c(\sigma^\circ}))
\end{equation*}
is convex and additionally, after slight thought, $$\partial^c(\overline{\partial^c\Psi^c(\sigma^\circ})^\circ) \subset \sigma.$$
By the second part of Theorem \ref{thm:SmoothSolutions} we thus obtain smoothness of the solution in $B\setminus\Sigma_\Psi$, and directly checking which part of $B$ for which we have shown smoothness one finds $\Sigma_\Psi=\Sigma_s$ as claimed.
\end{proof}

\begin{lemma}\label{lem:SymmetricCube}
Let $\Delta$ be the standard unit cube 
$$ \Delta = \conv\left\{ \pm e_0 \ldots \pm e_d \right\} $$
and $h=h_0$. Assume $\nu$ is invariant under all symmetries of $\Delta^\vee$ and don't charge $B_{d-1}$. Then \eqref{eq:MAInt} admits a solution. 
If in addition $\nu=\nu_N$ (or, more generally, for any $\sigma\in \mathcal{V}(\Delta^\vee)$, $(\beta_{\sigma})_\# \nu|_{\mathrm{St}(\sigma)}$ is absolutely continuous, and the density is smooth, uniformly non-negative and uniformly bounded) then the solution is smooth outside of the set
$$\Sigma_s := \{n\in B: \beta_j=0, |\beta_k|=|\beta_l|=\max_i |\beta_i|, j\neq k\neq l\}$$
where $\beta_0,\ldots,\beta_d$ are defined uniquely by writing $n=\sum_i \beta_if_i$ (when $d=1$, $\Sigma_s$ is empty and when $d=2$, $\Sigma_s$ is the midpoint of each of the edges). Additionally $\Sigma_s=\Sigma_\Psi$.
\end{lemma}
\begin{proof}
Note that 
$$ \Delta^\vee = \conv \left\{\pm f_0,\ldots,\pm f_d\right\}. $$

Let $I\in \{-1,1\}^{d+1}$ and $m_0=I_0e_0+\ldots+I_de_d$ be the corresponding vertex of $\Delta$. We will use $\SmSt(m_0)$ to denote the \emph{small closed star} of $m$
$$ \SmSt(m_0) = \left\{ m=\sum_{i=0}^d \alpha_i e_i \in \St(m_0), I_i\alpha_i\geq 0\forall i\right\}. $$
Similarly, for $j=0,\ldots,d$, we will use $\SmSt(f_j)$ to denote the \emph{small closed star} of $f_i$
$$ \SmSt(f_j) = \left\{ n=\sum_{i=0}^d \beta_i f_i \in \St(f_i), |\beta_j| \geq |\beta_i| \forall i\right\} $$
and $\SmSt(-f_j)$ to denote the \emph{small closed star} of $-f_i$
$$ \SmSt(-f_j) = \left\{ n=\sum_{i=0}^d \beta_i f_i \in \St(-f_i), |\beta_j| \geq |\beta_i| \forall i\right\}. $$
We claim that $\partial^c\Psi(\tau_{m_0}^\circ)\subset  \SmSt(m_0)$ for all vertices $m_0$ of $\Delta$ and $\partial^c\Psi^c(\sigma_{f_j}^\circ)\subset \SmSt(f_j)$ and $\partial^c\Psi^c(\sigma_{-f_j}^\circ)\subset \SmSt(-f_j)$ for all $j=0,\ldots,d$. To see this, assume first that $(m,n)$ contradicts the first claim. By symmetry, we may without loss of generality assume $m_0 = \sum_{i=0}^d e_i$. Hence $n=\sum_{i=0}^d\beta_i f_i$ for $\beta_i>0$ and $m=\sum_{i=0}^d\alpha_i e_i$ where $\alpha_j<0$ for some $j$. Let $F_j$ be the symmetry of $\Delta$ which maps $e_j$ to $-e_j$ and preserves $e_i$ for all $i\not= j$. Similarly, let $G_j$ be the symmetry of $\Delta^\vee$ which maps $f_j$ to $-f_j$ and preserves $f_i$ for all $i\not= j$. By symmetry, $(F_j(m),G_j(n))$ lies on the graph of $\partial^c\Psi$. We get 
\begin{eqnarray*}
    \langle F_j(m)-m,G_j(n)-n \rangle & = & \langle -2\alpha_je_j,-2\beta_jf_j \rangle \\
    & = & 4\alpha_j\beta_j \\
    & < & 0
\end{eqnarray*}
contradicting monotonicity of $\partial^c\Psi$. Assume instead $(m,n)$ contradicts the second claim. By symmetry, we may without loss of generality assume $j=0$. Hence $m=e_0+\sum_{i=1}^d \alpha_ie_i$ for $\alpha_i\in (-1,1)$ and $n=\sum_{i=0}^d \beta_if_i$ where either $|\beta_j|>|\beta_0|$ for some $j$ or $n\in \St(f_0)$ and consequently $\beta_0<0$. In the first case, let $F_{0j}$ and $G_{0j}$ be the symmetries of $\Delta$ and $\Delta^\vee$ which
interchanges the pairs $\{e_0,\sgn(\beta_j)e_j\}$ and $\{e_0,\sgn(\beta_j)e_j\}$, where $\sgn(\beta_j)$ is the sign of $\beta_j$, but leaves all other generators fixed. We get 
\begin{eqnarray*}
    F_{0j}(m)-m & = & \sgn(\beta_j)\alpha_je_0-e_0+e_j-\sgn(\beta_j)\alpha_je_j \\
    & = & (\sgn(\beta_j)\alpha_j-1)(e_0-e_j) 
\end{eqnarray*}
and 
\begin{eqnarray*}
    G_{0j}(n)-n & = & \sgn(\beta_j)\beta_jf_0-\beta_0f_0+\beta_0f_j-\sgn(\beta_j)\beta_jf_j  \\
    & = &  (\sgn(\beta_j)\beta_j-\beta_0)(f_0-f_j),
\end{eqnarray*}
hence    
\begin{eqnarray*}
    \langle F_{0j}(m)-m,G_{0j}(n)-n \rangle & = & (\sgn(\beta_j)\alpha_j-1)(\sgn(\beta_j)\beta_j-\beta_0)\langle e_0-e_j,f_0-f_j\rangle \\
    & < & 0.
\end{eqnarray*}
contradicting monotonicity of $\partial^c\Psi$.
In the second case, we have
\begin{eqnarray*}
    \langle F_0(m)-m,G_0(n)-n \rangle & = & \langle -2e_0,-2\beta_0f_0 \rangle \\
    & = & 4\beta_0 \\
    & < & 0,
\end{eqnarray*}
where, as above, $F_0$ and $G_0$ are the symmetries that map $e_0$ to $-e_0$ and $f_0$ to $-f_0$ while preserving $e_i$ and $f_i$ for all $i\not=0$. The third claim follows in the same way as the second claim. Applying Theorem~\ref{thm:SmoothSolutions} then proves the lemma. 
\end{proof}

In the lemma below, we fix a given presentation of a reflexive polytope $\Delta$ in the $\mathbb{Z}^{d+1}$-lattice. Given a vertex $n$ of $\Delta$ we define the \emph{small star}
\begin{equation}
    \mathrm{SmSt}(n) := \{x\in \mathrm{St}(n):||x-n||\leq||x-n_j|| \forall j\}.
\end{equation}

Here $\{n_j\}_j$ are the vertices of $\Delta$ and $||\cdot||$ is the Euclidean distance on $\mathbb Z^{d+1}\otimes \mathbb R=\mathbb R^{d+1}$. The presentation is fixed so that the Euclidean distance is well defined. However, note that this definition is consistent with the definitions of the small star in Lemma \ref{lem:Symmetric} and \ref{lem:SymmetricCube} with the given presentations after identifying the abstract lattice with $\mathbb{Z}^{d+1}$. In this situation we have the following lemma, completing the proofs of Lemma \ref{lem:Symmetric} and \ref{lem:SymmetricCube}. 

\begin{lemma}\label{lem:convexity in coords}
Let $\Delta$ be a lattice polytope in $\mathbb R^{d+1}$ satisfying the following assumption. For any vertex $v$ of $\Delta$, $v$ is the only vertex contained in the open half-plane $\{x:\langle v,x\rangle >0\}$, where $\langle \cdot,\cdot\rangle$ denotes the standard inner product on $\mathbb R^d$. 
Then the image of $\mathrm{SmSt}(v)$ for a vertex $v$ under the projection $\pi_v:\mathbb R^{d+1}\rightarrow \mathbb R^{d+1}/\mathbb{R}v$ is convex. 
\end{lemma}

\begin{proof}
    Let $v_0,...,v_N$ be the vertices of $\Delta$ and without loss of generality let $v=v_0$. Denote by $N$ the number of vertices of $\Delta$. Define $$C:=\{x\in \Delta:|x-v_0|\leq |x-v_j|\forall j\}$$
    where $|\cdot|$ is the standard norm on $\mathbb R^d$. 
    This set is clearly convex and thus so is its image under the projection. We aim to define a map from $C$ to $\mathrm{SmSt}(v_0)$ which becomes trivial after the projection. For any $p\in C$, define $F(p):=p+sv_0$, where $s$ is the maximal number such that $F(p)\in \Delta$. 
    
    Claim 1: $F$ decreases the distance to $v_0$. To see this, write $p=b_0v_0+\sum_{i=1}^{d}b_ie_i$ where $e_i$ is such that $(v_0,e_1,...,e_{d})$ is an orthonormal basis. Then $|p-v_0|^2=|1-b_0|^2+\sum_{i=1}^d |b_i|^2$ while $|F(p)-v_0|^2=|1-b_0-s|^2+\sum_{i=1}^d|b_i|^2$. Then simply note that $1-b_0\geq1-b_0-s\geq 0$. This uses the fact that the open half-space ${x:\langle v,x-v\rangle >0}$ does not contain any of the vertices of $\Delta$ which follows from the assumptions.
    
    Claim 2: $F$ increases the distance to all other vertices. To see this, pick some vertex $v_k$ different from $v_0$ and write $v_k=c_0n_0+\sum_{i=1}^dc_ie_i$. Then writing $p$ as before in the chosen basis we have $|p-v_k|^2=|b_0-c_0|^2+\sum_{i=1}^d |b_0-c_0|^2$ while we have $|F(p)-v_k|^2 = |b_0+s-c_0|^2+\sum_{i=1}^d|b_0-c_0|^2$. The claim follows after noting that $b_0+s_0-c_0\geq b_0-c_0\geq0$. 
    
We conclude that $F(p)\in C$, still, additionally, $F(p)$ lies in $\partial\Delta$ so that in fact $F(p)\in \mathrm{SmSt}(v_0)$.
Since $\pi_v \circ F=\pi_v$ we conclude that $\pi_v(C)\subset \pi_v(\mathrm{SmSt}(v_0))$. Since $C\subset \mathrm{SmSt}(v_0)$ trivially we find that $\pi_v(\mathrm{SmSt}(v_0))=\pi_v(C)$ is convex.
    
\end{proof}

\begin{remark}
    Existence of a weak solution when $\Delta$ is the unit cube and $\nu=\nu_N$ also follows from the result in \cite{LiToric}. 
\end{remark}

\subsection{Structural unstability and non-existence}\label{sec:Examples non-Existence}
We now turn to examples where Definition~\ref{def:Stability} provides an obstruction to existence of solutions. The main idea is illustrated in the following example.  
\begin{lemma}\label{lem:PointMass}
Assume $\Delta$ be the standard unit simplex (as in Lemma~\ref{lem:Symmetric}) and $h=h_0$. Let $n$ be a point in the interior of a facet $\tau$ of $\Delta^\vee_h$ and $\nu=\delta_n$ be a point mass at $n$. Then \eqref{eq:MAInt} does not admit a solution. 
\end{lemma}
\begin{proof}
We will prove that $(\Delta,h_0,\nu)$ is not stable. Assume $\gamma$ is a transport plan. It follows that
$$ \gamma(A\times \tau^\circ) \geq \gamma(A\times \{n\}) = \mu(A) = 1. $$
On the other hand, the intersection of $A\times \tau^\circ$ and \eqref{eq:WeaklyStableSet} is contained in $\St(m_\tau)\times \tau^\circ$, hence
\begin{eqnarray*} \gamma((A\times \tau^\circ)\cap \eqref{eq:WeaklyStableSet}) & \leq & \gamma(\St(m_\tau)\times \tau^\circ) \\
& \leq & \gamma(\St(m_\tau)\times B) \\
& = & \mu_M(\St(m_\tau)) \\
& < & 1. 
\end{eqnarray*}
It follows that $\gamma$ is not supported on \eqref{eq:WeaklyStableSet}, hence $(\Delta,h_0,\nu)$ is not stable.
\end{proof}

An alternative proof of Lemma~\ref{lem:PointMass} can be given by showing that the only admissible plan is $\gamma = \mu\times \delta_n$, which is not supported on \eqref{eq:WeaklyStableSet}. However, we choose the proof above since it highlights the more general obstruction that is manifest when $\nu$ assigns too much (or not enough) mass to some facet. In these cases no transport plan (regardless of optimality) is supported on \eqref{eq:WeaklyStableSet}. This idea will be the main tool to prove non-existence of solutions to \eqref{eq:MAInt}. In the rest of this section, we will focus on the case $\nu=\nu_N$. 
\begin{definition}\label{def:Structural Instability}
Let $\Delta$ be a reflexive polytope and $h:\Delta\cap M \rightarrow \mathbb Z$ a height function. We call the data $(\Delta, h)$ structurally unstable if at least one of the following holds: 
\begin{itemize}
\item There is some facet $\tau$ of $\Delta^\vee_h$ such that $\nu_N(\tau) > \mu_M(\mathrm{St}(m_\tau))$
\item There is some facet $\sigma$ of $\Delta$ such that $\mu_M(\sigma) > \nu_N(\cup_{m\in \sigma\cap M} \tau_m)$ 
\end{itemize}
We will write that a reflexive polytope is structurally unstable if $(\Delta,h_0)$ is structurally unstable, where $h_0$ is the trivial height function.
\end{definition}

Structural instability is stronger than instability.
\begin{lemma}
Let $\nu=\nu_N$. If 
$(\Delta,h)$
%$(\Delta, \Delta_h^\vee, \nu)$ 
is structurally unstable, then \eqref{eq:MAInt} does not admit a solution.
\label{lemma: structuraly unstability implies unstability}
\end{lemma}

\begin{proof}
It suffices to prove that if $(\Delta,h)$ is structurally unstable, then it is not stable. 
To that end, let $\gamma$ be any transport plan between $\nu_N$ and $\mu_M$ supported on 
\begin{equation*} 
\cup_{m\in A\cap M} \left(\St(m) \times \tau_m\right).
\end{equation*}
Let $\tau$ be any facet of $\Delta^\vee_h$ and consider the measure $\eta(\cdot) = \gamma(\cdot\times \tau)$ on $A$. By assumption $\eta$ is supported on $\mathrm{St}(m_\tau)$. The total mass of $\eta$ is $\nu_N(\tau)$ and thus $\eta(\mathrm{St}(m_\tau)) = \nu_N(\tau)$. But we also have $\eta \leq \mu_M$ in the sense of measures and thus $\eta(\mathrm{St}(m_\tau)) \leq \mu_M(\mathrm{St}(m_\tau))$ leading to 
\begin{equation*}
    \nu_N(\tau) \leq \mu_M(\mathrm{St}(m_\tau))
\end{equation*}
for any facet $\tau$ of $\Delta^\vee_h$. Next note that 
\begin{equation*}
    \cup_{m\in A\cap M} \left(\St(m) \times \tau_m\right) = \cup_{n\in \mathcal B_0} \sigma_n \times \cup_{m \in \sigma_n \cap M} \tau_m.
    %TODO: Jakob kolla gärna om du håller med om den här likheten. -- Det gör jag. / Jakob
\end{equation*}
Using this, analogously to the argument above one can prove that for any facet $\sigma$ of $\Delta$
\begin{equation*}
    \mu_M(\sigma) \leq \nu_N(\cup_{m\in \sigma\cap M} \tau_m).
\end{equation*}
\end{proof}

\begin{example}\label{ex:UnitSimplex}
Let $\Delta$ be the moment polytope of $\mathbb{P}^2$ and hence 
$$\Delta = \mathrm{conv}\{-e_0+2e_1,2e_0-e_1,-e_0-e_1\}.$$ 
Let $h_0$ be the trivial height function and define $h:\Delta\cap M\rightarrow \mathbb{Z}$ by
$$ h(n) = \max \{h_0(n), \langle (-1,4),n \rangle, \langle (1,5),n \rangle\},$$
i.e. $h(0,0)=0$ and
\begin{eqnarray*}
 & & h(-e_0+2e_1)=9  \\
 & & h(-e_0+e_1)=5, \,   h(e_1)=5  \\
 & & h(-e_0)=1, \, h(e_0)=1  \\
 & & h(-e_0-e_1)=1, \, h(-e_1)=1, \, h(e_0-e_1)=1, \, h(2e_0-e_1)=1
\end{eqnarray*}
Then $h$ extends to a piecewise affine convex function on $N_\mathbb R$ whose non-differentiable locus define a triangulation of $\Delta$. One finds that $$ \Delta_h^\vee = \mathrm{conv}\{f_0+5f_1,f_0+f_1,-f_1,-f_0,-f_0+4f_1\}.$$ The situation is depicted in Figure~\ref{fig: unstable with height 1d}.
\label{ex: unstable 1d nontrivial height}
\end{example}

% \begin{figure}
% \centering
% \begin{subfigure}{0.5\textwidth}
%   \centering
%   \includegraphics[width=\textwidth]{example2_primary.png}
%   \captionof{figure}{The primary polytope $\Delta$}
% \end{subfigure}%
% \begin{subfigure}{.5\textwidth}
%   \centering
%   \includegraphics[width=\textwidth]{example2_dual.png}
%   \captionof{figure}{The dual polytope $\Delta^\vee_h$}
% \end{subfigure}
% \caption{The polytope $\Delta$ from Example \ref{ex: unstable 1d nontrivial height} and its dual $\Delta^\vee_h$ relative to $h$ are shown in the subfigures (A) and (B) respectively. In (A), the numbers indicate the values of $h$ at the nearby lattice points. The dotted lines show the induced triangulation. The fatted line segment correspond to $\mathrm{St}(-e_0)$. In (B), the fatted line segment correspond to the face $\tau_{-e_0}$. The little circle depicts the origin in both figures. } 
% \label{fig: unstable with height 1d}
% \end{figure}

\begin{figure*}[h!]
    \subfloat[The primary polytope $\Delta$]{%
        \includegraphics[width=0.5\textwidth]{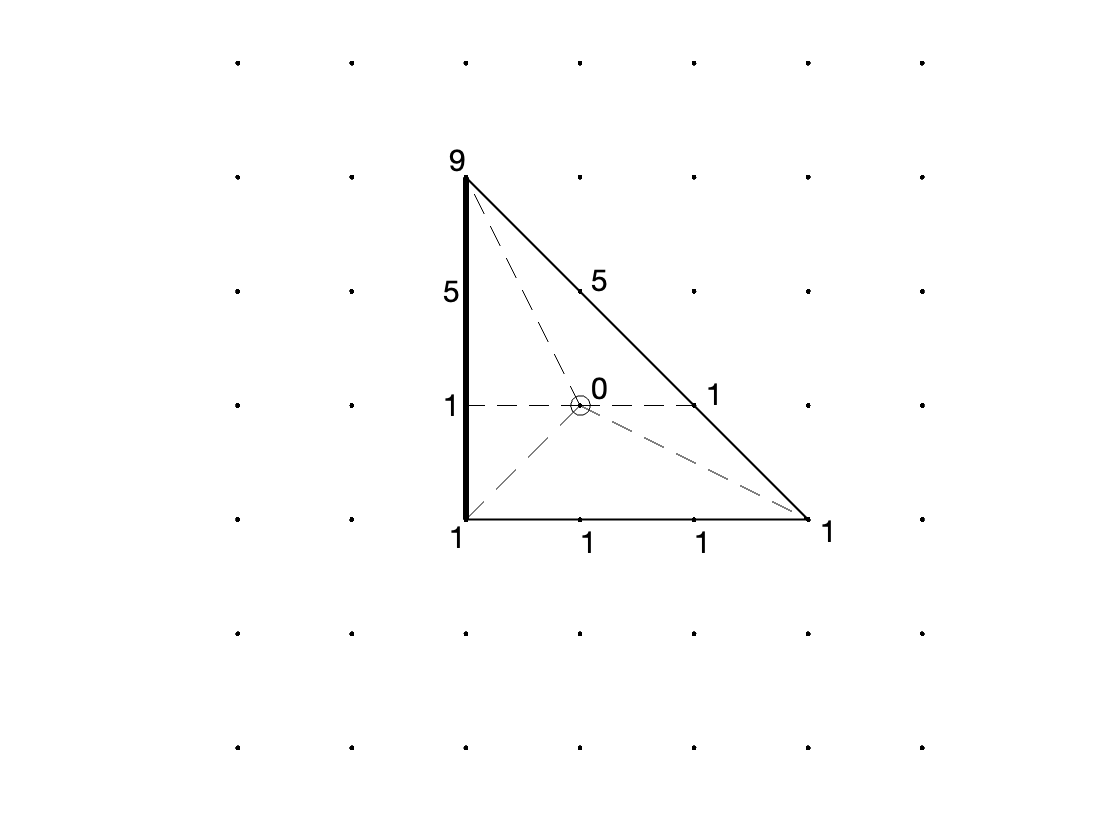}}\hfill
    \subfloat[The dual polytope $\Delta^\vee_h$]{%
        \includegraphics[width=0.5\textwidth]{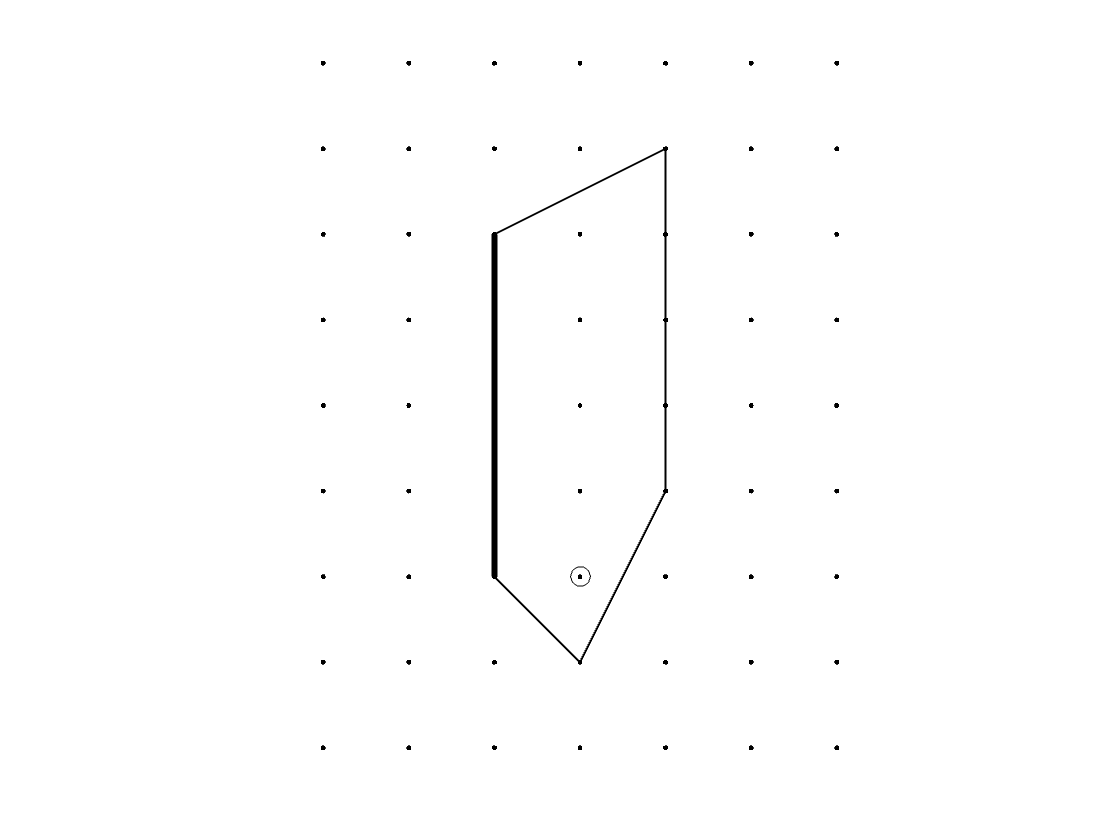}}
  
    \caption{The polytope $\Delta$ from Example \ref{ex: unstable 1d nontrivial height} and its dual $\Delta^\vee_h$ relative to $h$ are shown in the subfigures (A) and (B) respectively. In (A), the numbers indicate the values of $h$ at the nearby lattice points. The dotted lines show the induced triangulation. The fatted line segment correspond to $\mathrm{St}(-e_0)$. In (B), the fatted line segment correspond to the face $\tau_{-e_0}$. The little circle depicts the origin in both figures. } 
  
    \label{fig: unstable with height 1d}

\end{figure*}

\begin{lemma}\label{lem:d1 non-trivial height}
Let $\Delta$ and $h$ be as in Example~\ref{ex:UnitSimplex} and $\nu = \nu_N$. Then \eqref{eq:MAInt} does not admit a solution. 
\end{lemma}
\begin{proof} 
We check the structural stability condition for the facet $\tau_{-e_0}$ of $\Delta^\vee_h$ dual to $-e_0\in \Delta$. The associated $\mathrm{St}(-e_0)$ is just the single facet of $\Delta$ containing $-e_0$, indicatd by one of the fatted black lines in figure~\ref{fig: unstable with height 1d}. Computing the relevant volumes we end up with
\begin{equation*}
    \nu_N(\tau_{-e_0})=\frac{4}{11} > \frac{1}{3}=\mu_M(\mathrm{St}(-e_0))
\end{equation*}

and thus $(\Delta, h)$ is structurally unstable.

\end{proof}

\begin{remark}
In general, equation~\eqref{eq:MAInt} can be relaxed by considering general polarizations of $Y$. Example~\ref{ex: unstable 1d nontrivial height} and Lemma~\ref{lem:d1 non-trivial height} are especially interesting since they represent a case when considering general polarizations of $Y$ give no additional freedom. 
\end{remark}

Next are two structurally unstable examples with trivial height functions.  
\begin{example}
Let 
$$\Delta = \mathrm{conv}\{e_0+e_1,e_0,-e_1,-e_0+e_1\}$$ 
and $h=h_0$, hence 
$$\Delta^\vee_h = \Delta^\vee = \mathrm{conv}\{f_1,f_0,f_0-f_1,-2f_0-f_1\}.$$
\label{ex: unstable trivial height d1}
\end{example}
\begin{lemma}
Let $\Delta$ and $h$ be as in Example \ref{ex: unstable trivial height d1} and $\nu=\nu_N$. Then equation \eqref{eq:MAInt} does not admit a solution. 
\end{lemma}

\begin{proof}
Consider the vertex $v=-e_1$ of $\Delta$ and observe that

\begin{equation*}
    \nu_N(\tau_v) = 3/7 > 2/5 = \mu_M(\mathrm{St}(v)).
\end{equation*}
\end{proof}

\begin{example}
Let $\Delta$ be the 3-dimensional reflexive polytope
$$\Delta = \mathrm{conv}\{e_0,e_1,e_2,-2e_0-e_1-e_2,-e_0+e_1\}$$ 
(ID 16 in the database of reflexive polytopes in dimension $3$ by Kreuzer-Skarke \cite{KS}) and $h=h_0$, hence $$\Delta^\vee = \mathrm{conv}\{-f_0+f_2,f_1-2f_2,f_1+f_2,f_0+f_1+f_2,f_0-4f_1+f_2,f_0+f_1-4f_2\}. $$

\label{ex: structurally unstable in d2}

\end{example}

\begin{lemma}
Let $\Delta$ and $h$ be as in Example \ref{ex: structurally unstable in d2} and $\nu=\nu_N$. Then \eqref{eq:MAInt} does not admit a solution. 
\end{lemma}

\begin{proof}

Computing the volumes of the facets can be done by finding unimodular triangulations, i.e. triangulations made up entirely of simplices whose vertices make up a $\mathbb Z$-basis of the lattice. Then $\mu_M(\sigma)$ for a facet $\sigma$ is simply the number of triangles in the triangulation of $\sigma$, divided by the total number of triangles in a unimodular triangulation of $A$.  The volume of a facet of $\Delta^\vee$ can be computed in a similar manner and $\Delta^\vee$ with a triangulation of $\tau_{e_0}$ is shown in Figure \ref{fig: unstable d2}). 

Checking the structural stability condition for the facet $\tau_{e_0}$ of $\Delta^\vee$ 
we get

\begin{equation*}
\mu_M(\mathrm{St}(e_0)) = \frac{3}{8} < \frac{25}{58} = \nu_N(\tau_{e_0}).
\end{equation*}

\end{proof}

\begin{figure}[H]
    \centering
    \includegraphics[width = \textwidth]{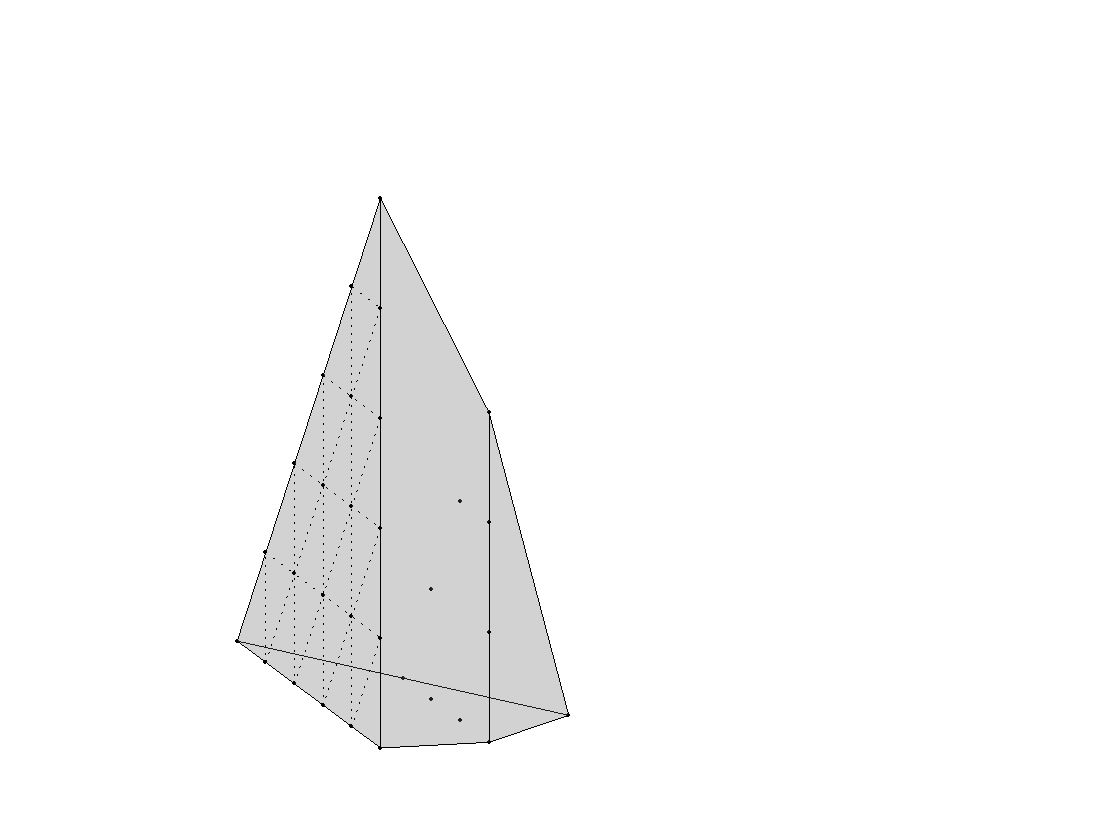}
    \caption{The polytope $\Delta^\vee$ from Example \ref{ex: structurally unstable in d2}. The 
    facet to the left in the picture, $\tau_{e_0}$, is triangulated with unimodular simplices. 
    Note that the number of simplices in the triangulation of $\tau_{e_0}$ is 25.}
    \label{fig: unstable d2}
\end{figure}

\subsection{Structural strict semistability and anomalous singularities}\label{sec:Examples Anomalous}

We now explore a limiting variant of structural unstability and show that it has consequences for the singularities of the proposed tropical affine structure. 

\begin{definition}\label{def: strict strictural semistability}
Let $\Delta$ be a reflexive polytope and $h:\Delta\cap M\rightarrow \mathbb Z$ a height function. We will call the data $(\Delta,h)$ \emph{structurally strictly semistable} if it is not structurally unstable but one of the following hold:
\begin{itemize}
\item There is some facet $\tau$ of $\Delta^\vee_h$ such that $\nu_N(\tau) = \mu_M(\mathrm{St}(m_\tau))$
\item There is some facet $\sigma$ of $\Delta$ such that $\mu_M(\sigma) = \nu_N(\cup_{m\in \sigma\cap M} \tau_m)$ 
\end{itemize}
We will write that a reflexive polytope $\Delta$ is structurally strictly semistable if $(\Delta,h_0)$ is structurally strictly semistable.
\end{definition}

\begin{lemma}\label{lem:ssss implies not connected}
Assume $(\Delta,h)$ is structurally strictly semistable and \eqref{eq:MAInt} admits a solution $\Psi$. Then $B\setminus \Sigma_\Psi$ is not connected.   
\end{lemma}
\begin{proof}
There are two cases. Assume first that there is some facet $\sigma$ of $\Delta$ such that $\mu_M(\sigma) = \nu_N(\cup_{m\in \sigma\cap M} \tau_m)$. By Theorem \ref{thm:MAInt}, since equation \eqref{eq:MAInt} admits a solution, there exists an optimal transport plan $\gamma$ supported on
    \begin{equation*}
        \cup_{m\in A\cap M} \left(\St(m) \times \tau_m\right)=\cup_{n\in \mathcal B_0} \sigma_n \times \cup_{m \in \sigma_n \cap M} \tau_m. 
    \end{equation*}
Define the measure $\gamma_\sigma(\cdot) =\gamma(\sigma\times \cdot)$ on $B$. The total mass of $\gamma_\sigma$ is $\mu_M(\sigma)$ and $\mathrm{Supp}(\gamma_\sigma) \subset \cup_{m \in \sigma \cap M} \tau_m$. Thus
\begin{equation*}
    \mu_M(\sigma)=\gamma_\sigma(\cup_{m \in \sigma \cap M} \tau_m) \leq \gamma(A\times \cup_{m \in \sigma \cap M} \tau_m) = \nu_N(\cup_{m\in \sigma\cap M} \tau_m).
\end{equation*}
But the left and right hand sides are equal and thus we conclude that $\mathrm{Supp}(\gamma_\sigma) = \cup_{m \in \sigma \cap M} \tau_m$. 
Consequently, for any $n\in B\setminus (\cup_{m \in \sigma \cap M} \tau_m)^\circ$, there is $m\in A\setminus \sigma^\circ$ such than $(n,m)\in \mathrm{Supp}(\gamma) \subset \partial^c \Psi^c$. Thus
\begin{align*}
     &B\setminus (\cup_{m \in \sigma \cap M} \tau_m)^\circ \subset \partial^c\Psi^c(A\setminus \sigma^\circ).
\end{align*}
Recalling the definition of $U_\sigma$ from Definition \ref{def:Usigma} we find that $U_\sigma \subset (\cup_{m \in \sigma \cap M} \tau_m)^\circ$

With a similar argument as above, we can conclude that \begin{equation*}
    \mathrm{Supp}(\gamma_{A\setminus \sigma^\circ})=B\setminus (\cup_{m \in \sigma \cap M} \tau_m)^\circ.
\end{equation*} 
From this it follows for any facet $\sigma'\neq \sigma$ of $\Delta$, $\mathrm{Supp}(\gamma_{\sigma'}) \subset B\setminus (\cup_{m \in \sigma \cap M} \tau_m)^\circ$ and thus $(\cup_{m \in \sigma \cap M} \tau_m) \subset \mathrm{Supp}(\gamma_{A\setminus \sigma'})$. We conclude similarly to above that
\begin{equation*}
    (\cup_{m \in \sigma \cap M} \tau_m)\subset \partial^c\Psi^c(A\setminus \sigma'^\circ),
\end{equation*}
hence that $U_{\sigma'} \subset B\setminus (\cup_{m \in \sigma \cap M} \tau_m)$ for any other facet $\sigma'$. 

Since all the charts $U_\sigma$ and $U_{\sigma'}$ are open and disjoint we conclude that the boundary of $\cup_{m \in \sigma \cap M} \tau_m$ lies in the singular set
\begin{equation*}
    \Sigma = B_{n-1}\setminus (\cup_{\sigma'\in \mathcal A_d}  U_{\sigma'}).
\end{equation*}
In conclusion, $B\setminus \Sigma$ is not connected. 

The other case, when there exists a facet $\tau$ of $\Delta^\vee_h$ such that $\nu_N(\tau) = \mu_M(\mathrm{St}(m_\tau))$ follows in a similar manner. More precisely, a similar argument as above gives 
$$ B\setminus \tau \subset \partial^c\Psi^c(A\setminus \St(m_\tau)^\circ), $$
hence $U_\sigma\subset \tau$ for all facets $\sigma$ contained in $\St(m_\tau)$, and 
$$ \tau^\circ \subset \partial^c\Psi^c(\St(m_\tau)), $$
hence $U_{\sigma'}\subset B\setminus \tau^\circ$ for all facets $\sigma'$ not contained in $\St(m_\tau)$. It follows that the boundary of $\tau$ lies in $\Sigma$, hence $B\setminus \Sigma$ is not connected. 
\end{proof}

As we will see, there are plenty of structurally strictly semistable polytopes among the reflexive polytopes in dimension 2 and 3. But in general, it might be difficult to check that they are stable and hence admit a solution to \eqref{eq:MAInt}. However \cite{LiToric} introduced a condition that is simple to check which imply existence of a solution to the Monge-Ampère equation \eqref{eq:MAInt}.

\begin{definition}\label{def:Li Admissable}
We will say that a reflexive polytope $\Delta$ is admissable in the sense of Li if there is no pair of vertices $v\in \Delta$ and $w \in \Delta^\vee$ such that $\langle v, w \rangle = 0$. 
\end{definition}

\begin{example}
Let $$\Delta = \conv\{e_0,e_1,e_2,-3e_0-e_1-e_2\}$$ (ID 2 in the classification \cite{KS} of three-dimensional reflexive polytopes ) and hence $$\Delta^\vee = \conv\{-f_0+f_1+f_2,f_0-5f_1+f_2,f_0+f_1-5f_2,+f_0+f_1+f_2\}.$$
\label{ex: admissable and semistable}
\end{example}
\begin{lemma}
Let $\Delta$ be as in Example \ref{ex: admissable and semistable}. Then $\Delta$ with height $h_0$ admits a solution to the Monge-Ampère equation \eqref{eq:MAInt} but  $B\setminus \Sigma$ is disconnected.
\label{lem: admissable and semistable}
\end{lemma}

\begin{proof}
Going through the possible pairs $(v,w)$ of vertices $v\in \Delta$ and $w\in \Delta^\vee$ one finds that $\Delta$ is admissible in the sense of Li. Thus by Theorem~2.31 in \cite{LiToric}, \eqref{eq:MAInt} admits a solution. Computing the volumes is analogous to Example \ref{ex: structurally unstable in d2} and we find
\begin{equation*}
    \nu_N(\tau_{e_0}) = 36/72 = 3/6 = \mu_M(\mathrm{St}(e_0)).
\end{equation*}
Thus $\Delta$ is structurally strictly semistable and by Lemma~\ref{lem:ssss implies not connected}, $B\setminus \Sigma$ is disconnected.
\end{proof}

\begin{remark}
Out of the 4319 3-dimensional polytopes that are reflexive, 145 are structurally strictly semistable and admissible in the sense of Li, hence express the same anomalous singular sets as Example~\ref{ex: admissable and semistable}.
\end{remark}
\begin{remark}\label{rem:Removable Singularity}
As explained in the introduction, it is natural to ask if the %perhaps 
surprisingly large
set $\Sigma_\Psi$ in Example~\ref{ex: admissable and semistable} can be taken to be any smaller, while preserving the existence of a solution to the Monge-Ampère equation in Definition~\ref{def:PDEproblem}. Recall that we do not know if the set $B\setminus \Sigma$ is the largest possible domain of a tropical affine structure on which one can extend the solution to the Monge-Ampère equation. For the polytope in Example~\ref{ex: admissable and semistable}, under natural assumptions on the 
solution $\Psi$,
it seems plausible that one can extend the regular set of the tropical structure so that the singularities are of codimension 2. However, the singular set will contain three of the vertices of $\Delta^\vee$, contrary to the expectation 
that the singularities are located in the interior of the edges (cf the case of the standard simplex \cite{LiFermat,HJMM}). We will now briefly explain this. 

First of all, tracing through the proof of Lemma~\ref{lem:ssss implies not connected}, one finds that $\tau_{e_0}^\circ \subset \partial^c\Psi^c(\mathrm{St}(e_0)^\circ)$ and $$ B\setminus\tau_{e_0} \subset \partial^c\Psi^c(A\setminus \mathrm{St}(e_0))=\partial^c\Psi^c(\sigma_{-f_0+f_1+f_2}^\circ).$$ 
Let $U$ be a neighborhood of one of the vertices of $\tau_{e_0}$. It follows that $U$ contains $n'$ such that $\partial^c\Psi(n')$ intersects $\mathrm{St}(e_0)^\circ$ and $n''$ such that $\partial^c\Psi(n')$ intersects $\sigma_{-f_0+f_1+f_2}^\circ$. It follows as in the proof of Lemma~\ref{lem:SingularOutsideCharts} that 
$ (\Psi-m_0)\circ\beta_\sigma^{-1} $
is not differentiable at either $n'$ or $n''$ for any $\sigma$ and $m_0\in \sigma$, hence the tropical affine structure and the solution to the Monge-Ampère equation can not be extended over any of the vertices of $\tau_{e_0}$. 

On the other hand, assuming $\partial^c\Psi$ is a homeomorphism, it is plausible that each edge of $\tau_{e_0}$ contains one point which is mapped to a vertex of $\Delta^\vee$ and that this point divides $\tau_{e_0}$ into two parts, each mapped into one facet of $\Delta^\vee$. Assume one of these parts is mapped into $\sigma$. Then, using Lemma~\ref{lem:GradientsEdge} as in the proof of Theorem~\ref{thm:MA}, $U_\sigma$ can be extended over this part. Consequently, if the assumptions made in the beginning of this paragraph holds then we arrive at a tropical affine structure with one singular point located at each vertex of $\tau$ and one singular point located in the interior of each edge of $\tau_{e_0}$. 
\end{remark}

\subsection{A classification of reflexive polygons and polyhedra}\label{sec:Computer Aided}
The reflexive polytopes are completely classified in dimension up to $4$ \cite{KS}. This puts them within reach of numerical computations and the condition of structural instability in addition to Li-admissibility can be checked numerically. This was done for reflexive polytopes of dimension 2 and 3 by the first author using Sage, with the results displayed in Table~\ref{tab:d1} and Table~\ref{tab:d2}. In particular, more than one third of the reflexive polytopes in dimension 3 are structurally unstable and does not admit a solution to \eqref{eq:MAInt}. As expressed in the table below, approximately one tenth of the reflexive polytopes in dimension 3 are structurally strictly semistable. Half of these are admissible in the sense of Li, hence admit solutions with anomalous singular set as in Example~\ref{ex: admissable and semistable} and Lemma~\ref{lem: admissable and semistable}. 

\begin{table}[H]
\caption{\label{tab:d1}Table specifying the total number of 2-dimensional reflexive polytopes and how many of these satisfy the conditions discussed in this section. Structurally unstable polytopes do not admit a solution to \eqref{eq:MAInt}. Polytopes that are admissable in the sense of Li admit a solution to \eqref{eq:MAInt}. Structurally strictly semistable polytopes which are admissible in the sense of Li admits a solution $\Psi$ such that $B\setminus \Sigma_\Psi$ is not connected, as in Example~\ref{ex: admissable and semistable} and Lemma~\ref{lem: admissable and semistable}.}
\begin{center}
\begin{tabular}{ |c|c| } 
 \hline
 \textbf{d = 1} &   \\ 
 \hline
reflexive polytopes of dimension d+1 & 16 \\ 
\hline
 structurally strictly semistable & 5  \\
 \hline
 structurally unstable & 2  \\
 \hline
 admissible in the sense of Li & 7 \\ 
 \hline
 structurally strictly semistable and admissable in the sense of Li & 3 \\
 \hline
\end{tabular}
\end{center}
\end{table}

\begin{table}[H]
\caption{\label{tab:d2}Table specifying the total number of 3-dimensional reflexive polytopes and how many of these satisfy the conditions discussed in this section. Recall that structurally unstable polytopes do not admit solutions to \eqref{eq:MAInt}. Polytopes that are admissable in the sense of Li admit a solution to \eqref{eq:MAInt}. Structurally strictly semistable polytopes which are admissible in the sense of Li admits a solution $\Psi$ such that $B\setminus \Sigma_\Psi$ is not connected, as in Example~\ref{ex: admissable and semistable} and Lemma~\ref{lem: admissable and semistable}. The classification of reflexive polytopes in dimension~3 together with various tools for analyzing lattice polytopes were provided by the PALP package in Sage, while volume computations where made using the Normaliz backend. See the git repository at 'https://doi.org/10.5281/zenodo.7615747' for details. }
\begin{center}
\begin{tabular}{ |c|c| } 
 \hline
 \textbf{d = 2} &  \\ 
 \hline
reflexive polytopes of dimension d+1 & 4319 \\ 
\hline
 structurally strictly semistable & 461 \\
 \hline
 structurally unstable & 1542 \\
 \hline
 admissible in  the sense of Li & 238\\
 \hline
 structurally strictly semistable and admissable in the sense of Li & 145\\
 \hline
\end{tabular}
\end{center}
\end{table}

\begin{remark}
Among 2-dimensional polytopes, two Delzant polytopes are admissible in the sense of Li, namely $\mathbb P^2$ and $\mathbb P^1\times \mathbb P^1$. In dimension three, four Delzant polytopes (out of the total 18) are admissible in the sense of Li, namely $\mathbb P^3$, $\mathbb P^2\times \mathbb P^1$, $(\mathbb P^1)^3$ and $\mathrm{Bl}_0\mathbb P^3$, the last one being the blow-up of $\mathbb P^3$ in one toric fixed point. None of the Delzant polytopes in dimension two and three are structurally unstable nor structurally strictly semistable. The two structurally unstable reflexive polytopes in dimension two are given by the $(1,2)$-weighted blowup of $\mathbb P^2$ (Example~\ref{ex: unstable trivial height d1}) and its dual. 
\end{remark}

\appendix

\section{SYZ-fibration}
This appendix is devoted to proving:
\begin{theorem}
    Assume $(\Delta,h_0)$ is stable and that $\nu_N(\cup U_\sigma)=\nu_N(\Delta^\vee)$. Let $K_{faces}$ be a compact subset of $B^\circ$ such that $\Psi|_{\tau^\circ}$ is smooth in a neighbourhood of $K_{faces}$ and $\tilde U_{faces}=R_{\geq 0}K_{faces}$ be the cone generated by $K_{faces}$. For each facet $\sigma$ of $\Delta$ 
    %and each facet $\tau$ of $\Delta^\vee$, 
    let $K_\sigma$ be a compact subset of $U_\sigma$ 
    %and $K_\tau$ a compact subset of $\tau^\circ$ 
    such that $(\Psi-m)|_{\St(n_\sigma)}$ is smooth in a neighbourhood of $K_\sigma$ for some (and hence all) $m\in \sigma$. Let 
    $$ \tilde U_{stars} = \cup_\tau( [0,1]\cdot K_\sigma + \bR_{\geq 0}n_\sigma). $$
    Then for small $t$, $X_t$ admits a special Lagrangian torus fibration on 
    $$ \Log_t^{-1}\left(\tilde U_{faces} \cup \tilde U_{stars}\right). $$
\end{theorem}
\begin{remark}
    The assumption that $\nu_N(\cup U_\sigma)=\nu_N(\Delta^\vee)$ is probably artificial, but included to simplify the exposition. 
\end{remark}

\subsection{Log Map}
Assume $h=h_0$, and hence 
$$ X = \{(x,t)\in Y\times \mathbb C^*: f_0(x)+t\sum f_m(x)=0. $$
Let $T_\bC\subset Y$ be the complex $(d+1)$-torus. Throughout the appendix, we will let $s=\log|t|\gg 0$. There is a map $\Log_s:T_\bC\rightarrow N_\bR$ defined by mapping $x\in T_\bC$ to the unique $\Log_s(x)=n\in N_\bR$ such that 
$$ \langle m,\Log_s(x)\rangle = \frac{1}{s}\log|f_m(x)| $$
for all $m\in M$. If we fix generators $m_0,\ldots,m_d$ of $M$ these determines coordinates $(\langle m_0,\cdot \rangle,\ldots,\langle m_d,\cdot \rangle)$ on $N_\bR$ and 
$$ (z_0,\ldots,z_d) = (f_{m_0}, \ldots, f_{m_d}) $$
on $T_\bC$ and $\Log_s$ takes the form
$$ \Log_s(z_0,\ldots,z_d) = \frac{1}{s}(\log|z_0|,\ldots,\log|z_d|). $$
The images of $X_t$ under $\Log_s$ converge uniformly to a tropical hypersurface in $N_\mathbb R$. More precisely, let $\mathcal A$ be the non-affine locus of the Legendre transform of $h_0$, i.e.
$$ \mathcal A = \{n: \max_{m\in M\cap \Delta} \langle m,n \rangle -h_0(m) = \langle m_i,n \rangle -h_0(m_i) \text{ for two distinct } m_1,m_2\in M\cap \Delta\}.  $$
Its complement $N_\bR\setminus \mathcal A$ has exactly one bounded component and its boundary is 
$$  \{n\in \mathcal A: \max_{m\in M\cap \Delta} \langle m,n \rangle -h_0(m) = \langle 0,n \rangle - h_0(0) = 0 \rangle\} = \partial\Delta^\vee. $$
Each face $F$ in $\partial\Delta^\vee$ lies in the boundary of an unbounded component of $\mathcal A$ contained in the subspace spanned by the vertices of $F$. %$n_{1},\ldots,n_{l}$ where $\tau_1,\ldots,n_l$ are the faces of $\Delta^\vee$ containing $F$. 
\begin{lemma}(\cite[Proposition~3.2]{LiFermat}).\label{lem:Amaeba}
     \begin{eqnarray*}
         \dist_{N_\bR}(x,\mathcal A) & \leq & \frac{C}{s} \text{ for all } x\in \Log_s(X_t) \\
         \dist_{N_\bR}(\Log_s(X_t),x) & \leq & \frac{C}{s} \text{ for all } x\in \mathcal A. 
     \end{eqnarray*}
\end{lemma}
\begin{proof}
    The lemma is proved in the same way as when $\Delta$ is the unit simplex (see \cite[Proposition~3.2]{LiFermat}). A crucial point is that if the polynomial $f_0(x)+t\sum f_m(x)$ vanishes, then its two largest terms has to be of comparable size, which for small $t$ implies $\Log_s(x)$ is close to $\mathcal A$. 
\end{proof}

\subsection{Holomorphic Coordinate Charts}
For a facet $\sigma$ of $\Delta$ and $\delta\gg 0$, let 
$$ V^\sigma_{t,\delta} = \{ x\in X_t: t|f_m(x)|<e^{-\delta}|f_0(x)| \text{ for all } m\in M\cap \Delta \setminus (\sigma\cup\{0\})\}. $$
Pick $m_0,\ldots,m_d$ generating $M$ such that $\{m_1,\ldots,m_d\}\subset n_\sigma^\perp$ and let $z_0,\ldots,z_d$ be the corresponding coordinates
$z_i(x) = f_{m_i}(x)$. Differentiating $F_t:= f_0(x)+t\sum f_m(x)$ with respect to $z_0$ gives
$$ \frac{\partial F_t}{\partial z_0} = c_0f_{-m_0} + t\sum_{\{m\in V(\Delta)\setminus \sigma\}} c_m f_{m-m_0} $$
for constants $c_0$ and $c_m, m\in N\cap \Delta \setminus (\sigma\cup\{0\})$.
This is non-zero since $\left|\frac{f_{0}}{f_{m_0}}\right|>>t\left|\frac{f_{m}}{f_{m_0}}\right|$ for each $m\in N\cap \Delta \setminus (\sigma\cup\{0\})$. Consequently, the implicit function theorem furnishes coordinates 
$$ \tilde \beta^{-1}:(z_1,\ldots,z_d) \mapsto (z_0(z_1,\ldots,z_m),z_1,\ldots,z_m) $$
on $V^\sigma_{t,\delta}$. 

Moreover, 
%letting $s=-\log t$ and 
%$$ (x_0,\ldots,x_d) = \Log_s(x) := \frac{1}{s}(\log|z_0|,\ldots,\log|z_d|) $$
%we get that 
$x\in V^\sigma_{t,\delta}$ if and only if 
$$ -s+\langle m, sn \rangle < -\delta,  $$
or equivalently
$$ \langle m, n \rangle \leq 1-\delta/s, $$
for all $m\in N\cap \Delta \setminus (\sigma\cup\{0\})$. As $s\rightarrow +\infty$, this set covers larger and larger parts of the open star of $n_\sigma$. 

As in \cite{LiFermat} we will work with two different coordinate charts: Fix a small parameter $\rho>0$ and for $n\in N_\bR$, let $B_\rho(N)$ be the ball of radius $\rho$ (with respect to some norm on $N_\bR$) centered at $n$. 
\begin{itemize}
    \item For each facet $\sigma$ of $\Delta$, the \emph{starlike chart}
    $$ \tilde U^\sigma_{t,\delta,\rho} = \{ x\in V^\sigma_{t,\delta}: B_\rho(\Log_s(x))\in [0,1]\cdot K_\sigma + \bR_{\geq 0}n_\sigma) \}. $$
    \item For each facet $\sigma$ of $\Delta$ and vertex $m$ of $\sigma$, the facelike chart
    $$\tilde U^{\sigma,m}_{t,\delta} = \{ x\in V^\sigma_{t,\delta}: B_\rho(\Log_s(x))\in \mathbb R_{\geq 0} (\tau_m^\circ\cap K_{faces}) \}. $$
\end{itemize}

The holomorphic volume form on $V^\sigma_{t,\delta}\subset X_t$ is 
\begin{equation} \label{eq:HolVolumeForm} \Omega_t = \pm \frac{d\log z_0\wedge\ldots\wedge d\log z_d}{dF_t} \approx d\log z_1 \wedge \ldots \wedge d\log z_d. \end{equation}
Since $\nu_N(\cup U_\sigma) = \nu_N(\partial\Delta^\vee)$, we conclude that for small $t$ and $\rho$, the starlike charts account for almost all of the Calabi-Yau volume of $X_t$.

\subsection{Local torus action}
In addition to the action of $T_\bC$ on $Y$, there is a local $(\mathbb C^*)^d$-action on $\tilde V^\sigma_{t,\delta}\subset X_t$, defined in coordinates by
$$ (\lambda_1,\ldots,\lambda_d)\cdot (z_0(z_1,\ldots,z_d), z_1, \ldots, z_d) = (z_0(\lambda_1z_1,\ldots,\lambda_dz_d), \lambda_1z_1, \ldots, \lambda_d z_d). $$
The images under $\Log_s$ of $(S^1)^d$-orbits of this action are contained in sets of the form $n+\mathbb Rn_\sigma$ for $n\in N_\bR$. 

\subsection{Model Metric}\label{sec:ModelMetric}
Let $\nu=\nu_N$ and $\Psi$ be the solution to the Monge-Ampère equation in Definition~\ref{def:PDEproblem}. By standard theory, $(\Psi-m_0)\circ \Log_s$ defines a continuous semi-positive metric on $-K_Y$. Note also that $(\Psi-m_0)\circ \Log_s$ is Lipschitz continuous in logarithmic coordinates on $T_\bC$. 
\begin{lemma}
    Let $\sigma$ be a facet of $\Delta$ and $m$ a vertex of $\sigma$. Then, for small $t$
    $$ ((\Psi-m_0)\circ \Log_s)|_{\tilde U^\sigma_{t,\delta,\rho}} $$
    is $(S^1)^d$-invariant and its complex Monge-Ampère measure is 
    $$ d\log z_1\wedge d\log\bar z_1\wedge \ldots \wedge d\log z_d\wedge d\log \bar z_d \approx \Omega_t\wedge\overline{\Omega_t} $$
\end{lemma}
\begin{proof}
    We claim that for small $t$, $\tilde U^\sigma_{t,\delta}$ is contained in 
    \begin{equation} \label{eq:SigmaSet} \{n\in N_\bR: \Psi(n) = \langle m,n \rangle - \Psi^c(m) \text{ for some } m\in \sigma \}.\end{equation}
    To see this, note that if 
    $$ \Psi(n) = \sup_{m'\in \partial\Delta} \langle m',n \rangle - \Psi^c(m') =  \langle m,n \rangle - \Psi^c(m) $$
    for some $m\in \sigma$, then
    \begin{eqnarray*}
        \Psi(n+\lambda n_\sigma) & = & \sup_{m'\in \partial\Delta} \langle m',n+\lambda n_\sigma \rangle - \Psi^c(m') \\
        & = & \langle m,n + \lambda n_\sigma\rangle - \Psi^c(m)
    \end{eqnarray*} 
    for any $\lambda\geq 0$ since $\langle m,n_\sigma \rangle = 1 \geq \langle m',n_\sigma \rangle$ for all $m'\in \partial\Delta$. 

    The first part of the lemma then follows from Lemma~\ref{lem:Psi-mInvariance} below together with the fact that $K_\sigma$ is contained in the open set
    $$ \{n\in N_\bR: \Psi(n) > \langle m,n \rangle -\Psi(m) \text{ for all } m\notin \sigma^\circ\} $$
    which is contained in \eqref{eq:SigmaSet}. 
    
    The second part of the lemma follows from the first part and the classical correspondence between the real and complex Monge-Ampère operator for $(S^1)^d$-invariant metrics under a complex torus action. 
\end{proof}
In particular, since $\cup_\sigma U_{\sigma}$ covers all the mass of $\nu_N$ by assumption, it follows from \eqref{eq:HolVolumeForm} that the starlike charts $\cup_\sigma \tilde U^{\sigma}_{t,\delta,\rho}$ covers almost all the Monge-Ampère mass $\tilde\nu_\Psi$ of $(\Psi-m_0)\circ \Log_s$. 
In particular, it follows that the total variation of $\Omega_t\wedge \bar \Omega_t-\tilde\nu_\Psi$ is arbitrarily small for small $t$ and $\rho$
\begin{lemma}\label{lem:Psi-mInvariance}
    Let $n\in \partial^c\Psi^{-1}(\sigma)$ and $\lambda\in \bR$ and assume $n+\lambda n_\sigma\in \partial^c\Psi^{-1}(\sigma)$. Then 
    $$ \Psi(n+\lambda n_\sigma) = \Psi(n) + \lambda. $$
\end{lemma}
\begin{proof}
    Since $\langle m,n_\sigma \rangle =1$ for any $m\in  \sigma$, we have 
    \begin{eqnarray*}
        \Psi(n+\lambda n_\sigma) & = & \sup_{m\in \partial\Delta} \langle m,n +\lambda n_\sigma \rangle - \Psi^c(m) \\
        & = & \sup_{m\in \sigma} \langle m,n +\lambda n_\sigma \rangle - \Psi^c(m) \\
        & = & \sup_{m\in \sigma} \langle m,n \rangle - \Psi^c(m) + \lambda  \\
        & = & \Psi(n) + \lambda.
    \end{eqnarray*} 
\end{proof}

Finally, on the facelike charts $U_{t,\delta,\rho}^{\sigma,m}$, the map 
$ p_\sigma \circ\Log_s:U_{t,\delta,\rho}^{\sigma,m}\rightarrow \tau_m^\circ $
(c.f. Section~\ref{sec:Pairing}) can be used to pull back a local potential 
$\Psi\circ p_\sigma \circ\Log_s$ which by Lemma~\ref{lem:Amaeba} and the Lipschits bound on $\Psi$ differ from $\Psi \circ\Log_s$ by a $C^0$-perturbation.

\subsection{Local $C^0$-estimate}
As in \cite{LiSYZ}, the argument relies on the uniform Skoda estimate for maximally degenerate polarized families of Calabi-Yau manifolds in \cite{LiSkoda} and the following $L^1$-stability estimate:
\begin{theorem}(\cite[Theorem~2.6]{LiSYZ})\label{thm:L1Stability}
    Let $(Z,\omega)$ be a compact Kähler manifold and $u \in \PSH(Z,\omega)\cap C^0(Z)$. Assume 
    \begin{itemize}
        \item There is a Skoda estimate, i.e. constants $\alpha$ and $A$ such that
        \begin{equation} \int_Z e^{-\alpha v}\omega^d \leq A \text{ for all } v\in \PSH(Y,\omega), \sup v=0.  \end{equation}
        \item The subzero set $\{u\leq 0\}$ has a mass lower bound
        $$ \int_{\{u\leq 0\}} \omega^d \geq \lambda. $$
        \item The total variation of $\omega^d-(\omega+i\partial\bar\partial u)^d$ is bounded by $s^{2n+3}<1$ for some constant $s'$. 
        \item $u$ is smooth away from a closed subset $S$ with zero $\omega^d$-measure.
        \item $||u||_{C^0} \leq A'$. 
    \end{itemize}
    Then for small $s'$, there is a uniform estimate
    $$ \sup u \leq Cs' $$
    where $C$ only depends on $\lambda, d, \alpha, A, A'$.
\end{theorem}
Let $\omega_t=\omega_{FS}|_{X_t}/s$ where $\omega_{FS}$ is the Fubini study metric of some projective embedding induced by a multiple of $-K_Y$. Let $\omega_{CY,t}$ be the Calabi-Yau metric in the class of $\omega_{t}$ and $\phi_{CY,t}$ its potential with respect to $\omega_t$, i.e. $\omega_{CY,t} = \omega_t + i\partial\bar \partial \phi_{CY,t}$.  
Now, $X$ defines a polarized algebraic Calabi-Yau degeneration family in the sense of \cite{LiSkoda}, hence by \cite[Theorem~1.3]{LiSkoda}$, (X_t,\omega_t)$ admits a Skoda inequality with respect to $\omega_{CY,t}^d$ 
$$ \int_{X_t} e^{-\alpha v}\omega_{CY,t}^d \leq A \text{ for all } v\in \PSH(X_t,\omega_t), \sup v=0. $$
for constants $\alpha, A$ independent of $t$. As a consequence (see \cite[Theorem~1.4]{LiSkoda}), $\phi_{CY,t}$ satisfies an $L^\infty$-bound which is uniform in $t$. 
%$ ||\phi_{CY,t}||_{L^\infty}<C' $ for
%$C'$ independent of $t$. 
Adjusting the constant $A$ in a way that only depends on $||\phi_{CY,t}||_{L^\infty}$, we get a new Skoda inequality
$$ \int_{X_t} e^{-\alpha v}\omega_{CY,t}^d \leq A \text{ for all } v\in \PSH(X_t,\omega_{CY,t}), \sup v=0. $$
Finally, let $\psi_t$ satisfy $\omega_t+i\partial\bar\partial \psi_t = i\partial\bar\partial\left(\Psi\circ \Log_s\right)|_{X_t}$. Note that $||\psi_t||_{C^0(X_t)}\leq ||\psi||_{C^0(Y)}<+\infty$ and, by the discussion in Section~\ref{sec:ModelMetric}, the total variation of $(\omega_t+i\partial\bar\partial\phi_{CY,t})^d-(\omega_t+i\partial\bar\partial\psi_{t})^d$ is arbitrary small for small $t$ and $\rho$. The $L^1$-stability estimate above then have the following consequence:
\begin{theorem}\label{thm:CloseToSup}
$\psi_t-\phi_{t}^{CY}$ is close to its maximum with large probability, i.e. for all $\epsilon>0,\lambda>0$
$\tilde\nu_t(\{\psi_t-\phi_{t}^{CY}\leq -\epsilon\})<\lambda$ for small $t$. 
\end{theorem}
\begin{proof}
    Assume $c>0$ satisfies 
    $$ \mu_t(\{\psi_t-\phi_{t}^{CY}<-c\})\geq \lambda. $$
    Then by the $L^1$-stability estimate Theorem~\ref{thm:L1Stability} 
    $$ c=\sup \psi_t-\phi_t^{CY}+c. $$
    is bounded from above by something which vanishes as $t\rightarrow 0$, this proves the lemma. 
\end{proof}
On any holomorphic chart, Theorem~\ref{thm:CloseToSup} implies local $C_0$-convergence (on slightly smaller charts) by a simple application of the mean value inequality. 
\begin{corollary}
Let $K$ be a compact subset of a starlike chart $\tilde U^\sigma_{t,\delta,2\rho}$ or a facelike chart $\tilde U^{\sigma,m}_{t,\delta,2\rho}$, then 
$$ \sup_{\Log_t^{-1}(K)} |\psi_t-\phi_t^{CY}|\rightarrow 0 $$
as $t\rightarrow 0$.
\end{corollary}
\begin{proof}
    For $x\in U^\sigma_{t,\delta,2\rho}$, let $B_{\rho}(x)$ be the ball of radius $\rho$ centered at $x$ in logarithmic coordinates on $V^\sigma_{t,\delta}$ and note that $B_{\rho}(x)\subset U^\sigma_{t,\delta,\rho}$.
    %$\Log_s(B_{\rho}(x))\subset [0,1]\cdot U_\sigma+\bR_{\geq 0}n_\sigma$.
    %for any $x\in \Log_t^{-1}(K)$ and 
    Let $0<\epsilon<<1$ and $0<\lambda<<|B_\delta|$ and apply Theorem~\ref{thm:CloseToSup}. Note that since $\Psi\circ \Log_t$ satisfies a Lipschitz bound $L$, we get 
    $$ \psi(x) \geq -\delta L+\frac{1}{|B_\delta(x)|}\int_{B_\delta} \psi. $$
    By the mean value inequality and Theorem~\ref{thm:CloseToSup}, for small $t$
    \begin{eqnarray*}
        \psi(x)-\phi_t^{CY} & \geq & -\delta L + \frac{1}{|B_\delta(x)|}\int_{B_\delta(x)} \left(\psi(x)-\phi_t^{CY}\right) \\
        & \geq & -\delta L - \frac{\lambda}{|B_\delta(x)|}\left(||\psi||_{C^0}+||\phi_t^{CY}||_{C^0}\right) - (1-\lambda)\epsilon.
    \end{eqnarray*}
    The statement for facelike charts $U^{\sigma,m}_{t,\delta,2\rho}$ is proved in the same way. 
\end{proof}

\subsection{Higher Order Convergence and SYZ fibration}
From the $C^0$-estimates above we conclude that the Calabi-Yau metrics are $C^0$-approximated by $\Psi\circ \Log_s$ on starlike charts $U^{\sigma}_{t,\delta,2\rho}$ and $\Psi\circ p_\sigma \circ \Log_s$ on facelike charts $U^{\sigma,m}_{t,\delta,2\rho}$. Both of these are semi-flat and have Monge-Ampère measures which approximate the Calabi-Yau volume forms. The higher order estimates and existence of a special Lagrangian torus fibration then follow as in \cite[Section~5.2 and Section~5.4]{LiToric} from Savin's Small Perturbation Theorem \cite{Savin} and Zhang's result on special Lagrangian torus fibrations \cite{Zhang}.

\end{document}